\documentclass[12pt]{article}

\usepackage[paper=letterpaper,margin=0.7in,twoside=false,includehead]{geometry}

\usepackage{amsfonts,amsmath,amsthm, amssymb, thmtools, thm-restate} 
 
\usepackage{hyperref}
\usepackage[capitalize]{cleveref}

\usepackage{graphicx}
\usepackage{fancyhdr}
\usepackage{xparse}
\usepackage{mathtools}

\usepackage{verbatim}

\usepackage[shortlabels]{enumitem}
\usepackage[normalem]{ulem}

\SetEnumitemKey{thmparts}{topsep=2pt plus 1pt minus1pt, itemsep=1pt plus0.5pt minus0.5pt}

\usepackage[T1]{fontenc}

\usepackage{tikz}
\usepackage{tkz-berge}
\usetikzlibrary{decorations.pathreplacing}

\newcommand{\N}{\mathbb{N}}

\newcommand{\symd}{\bigtriangleup}
\DeclarePairedDelimiter\ceil{\lceil}{\rceil}
\DeclarePairedDelimiter\floor{\lfloor}{\rfloor}

\DeclarePairedDelimiter\abs{|}{|}
\DeclarePairedDelimiter\parens{(}{)}
\DeclarePairedDelimiter\set{\{}{\}}

\DeclarePairedDelimiterX\setof[2]{\{}{\}}{#1\,:\,#2}

\DeclareDocumentCommand{\k}{O{t}}{k^{#1}}

\DeclareMathOperator{\Ts}{T}
\DeclareMathOperator{\CTs}{CT}

\DeclareDocumentCommand{\G}{O{r}O{n}}{G^*_{#2,#1}}
\DeclareDocumentCommand{\Gs}{O{r}O{n}O{\ell}}{G^{#3*}_{#2,#1}}
\DeclareDocumentCommand{\T}{O{n}O{r}}{\Ts_{#2}(#1)}
\DeclareDocumentCommand{\CT}{O{m}O{r}}{\CTs_{#2}(#1)}

\newtheorem{thm}{Theorem}[section]
\newtheorem{lem}[thm]{Lemma}
\newtheorem{cor}[thm]{Corollary}

\newtheorem{prop}[thm]{Proposition}

\theoremstyle{definition}
\newtheorem{defn}[thm]{Definition}
\newtheorem{qu}{Question}
\newtheorem{rem}[thm]{Remark}
\newtheorem{obs}[thm]{Observation}
\newtheorem{case}{Case}
\numberwithin{case}{thm}

\numberwithin{subcase}{case}

\newcommand{\A}{\mathcal{A}}
\newcommand{\cC}{\mathcal{C}}

\DeclareDocumentCommand{\Gell}{O{r}O{n}O{\ell}}{\mathcal{G}_{#2,#1}^{#3}}
\DeclareDocumentCommand{\Hl}{O{r}O{n}O{\ell}}{\mathcal{H}_{#2,#1}^{#3}}
\DeclareDocumentCommand{\Hlhat}{O{r}O{n}O{\ell}}{\widehat{\mathcal{H}}_{#2,#1}^{#3}}

\begin{document}

\pagestyle{plain}

\title{Maximizing the number of cliques in $K_{r+1}$-free graphs with forbidden properties}
\author{A. Dawkins and R. Kirsch}
\date{September 21, 2026}
\maketitle
\abstract{
Ferrero and Lesniak in 2018 found the maximum numbers of edges in $r$-partite non-Hamiltonian graphs. Recently we found the maximum numbers of edges and $t$-cliques in $K_{r+1}$-free graphs (1) that are not Hamiltonian or (2) that satisfy a condition on low-degree vertices related to P\'{o}sa's theorem. Applying theorem (2), here we extend theorem (1) from Hamiltonicity to other properties. We determine the maximum numbers of edges and $t$-cliques in $K_{r+1}$-free graphs that avoid one of the following properties: traceability, Hamiltonian-connectedness, $k$-path Hamiltonicity, $k$-Hamiltonicity, $k$-Hamiltonian-connectedness, and $k$-connectedness. We find all extremal graphs having the maximum numbers of edges. On the way, we prove upper bounds on the numbers of edges and $t$-cliques in $K_{r+1}$-free graphs that avoid an arbitrary stable property that holds for sufficiently large complete graphs.
}

\section{Introduction}\label{sec:intro}

A graph is \emph{Hamiltonian} if it contains a cycle visiting each vertex exactly once. Due to the central importance of Hamiltonicity in extremal graph theory and the computational difficulty of deciding whether a given graph is Hamiltonian, a great deal of research aims to determine best-possible conditions sufficient to guarantee Hamiltonicity, or equivalently to determine extremal values of various parameters among non-Hamiltonian graphs \cite{KO14}. See \cite{G03, G14} for surveys.

\subsection{Sufficient conditions for Hamiltonicity in various graph classes}

Sufficient conditions for Hamiltonicity have been studied extensively, first in all graphs, next in bipartite graphs, and then in $r$-partite graphs for $2 < r < n$. Of particular relevance here, best-possible edge density conditions sufficient to imply Hamiltonicity (equivalently, extremal theorems on the maximum number of edges in non-Hamiltonian graphs) were determined first by Ore \cite{OreEdgeCond} in 1961 for all graphs, then by Moon and Moser \cite{MoonMoser} in 1963 for bipartite graphs, by Adamus \cite{Adamus} in 2009 for balanced tripartite graphs, and by Ferrero and Lesniak \cite{FerreroLesniak} in 2018 for $r$-partite graphs with given part sizes. The development of minimum degree conditions \cite{Dirac, MoonMoser, CFGJL} and Ore-type ($\sigma_2$) conditions \cite{Ore, MoonMoser, ChenJacobson} for Hamiltonicity and edge density conditions for long cycles \cite{Erdos,
AdamusUnbalanced, Araujo} followed similar progressions from general graphs, to bipartite graphs, to $r$-partite graphs.

In \cite{DK26} we relaxed the $r$-partite condition to a $K_{r+1}$-free condition and determined a best-possible edge density condition for Hamiltonicity. We write $e(G)$ for the number of edges in a graph $G$ and $\T[n][r]$ for the $n$-vertex, $r$-partite Tur\'{a}n graph.

\begin{thm}[Theorem 5.4 in \cite{DK26}]\label{theorem:kr+1ham}
Let $G$ be an $n$-vertex, $K_{r+1}$-free graph where $r\ge 3$.
If $G$ is not Hamiltonian and \[n \ge \begin{cases}26 & \text{if }r=3\\ 11 & \text{if } r=4\\
2 & \text{if } r \ge 5\end{cases},\quad\text{ then }e(G) \leq e(\T[n-1][r]) + 1.\] Equality holds if and only if $G$ is $\T[n-1][r]$ plus a pendant edge or $G$ is $K_{3,1,1}$, $K_{6,2,2,1}$, $K_{4,1,1,1}$, or $K_{5,1,1,1,1}$, with the exceptional graphs occurring in the cases where $r \ge 5$ and $n =5$, or $(r,n)$ is $(4,11)$, $(5,7)$, or $(5,9)$, respectively.
\end{thm}

\subsection{Sufficient conditions for stable properties}

In the 1960s and 70s a series of papers proved sufficient conditions for properties such as traceability \cite{OreEdgeCond}, Hamiltonian-connectedness \cite{Berge, OreHConnected, Williamson}, $k$-path Hamiltonicity \cite{KronkGen, Kronk}, $k$-Hamiltonicity \cite{CKL70,Chvatal}, $k$-Hamiltonian-connectedness \cite{Lick}, and $k$-connectedness \cite{Boesch74,Bondy69,CKK68}. The definitions of all of these properties, which we call \emph{Hamiltonicity-like}, appear in \cref{subsec:edgedegreestable}. These properties all are \emph{stable}, the definition of which also appears in \cref{subsec:edgedegreestable}. Then Bondy and Chv\'{a}tal \cite{BC76} proved in general form sufficient conditions for stable properties that hold for sufficiently large complete graphs, principally including \cref{thm:kstable}, which is based on having few low-degree vertices or many high-degree vertices (called \emph{Chv\'atal-type} for their development in \cite{Chvatal}). A Chv\'{a}tal-type condition implies a $\sigma_2$ condition, which in turn implies edge density and minimum degree conditions. Bauer et al. \cite{BBHKNSWY15} surveyed best-possible Chv\'{a}tal-type degree conditions for many properties.
    
Reflecting the development of sufficient conditions for Hamiltonicity, sufficient conditions for a few properties related to Hamiltonicity in bipartite graphs were determined in the 1980s and 90s \cite{AdamusBalanced,AdamusUnbalanced,BaggaVarma,BV91,BS81,ES88}. 
Bondy and Chv\'{a}tal \cite{BC76} briefly mentioned that a notion of stability and closure could be developed for balanced bipartite graphs, and Amar et al. in 1995 presented five types of theorems for twelve different bistable properties in balanced bipartite graphs \cite[Table 1]{AmarEtAl}. 

The first main results of this paper are best-possible sufficient edge density conditions for several different properties in $K_{r+1}$-free graphs, given in Theorems \ref{theorem:kr+1all} and  \ref{theorem:kr+1kpath1}. We characterize the extremal graphs. As the extremal graphs are $r$-partite (and every $r$-partite graph is $K_{r+1}$-free), these results also imply new, best-possible sufficient edge density conditions for the same properties in $r$-partite graphs.

An important ingredient in the proofs is \cref{thm:stability}, which gives an edge density condition in $K_{r+1}$-free graphs sufficient to imply an arbitrary stable property that holds for sufficiently large complete graphs.

\subsection{Clique density conditions}
    
In the last several years, motivated by the study of generalized Tur\'{a}n problems (see \cite{GP25} for a survey), a different type of sufficient condition has emerged based on the number of copies of a subgraph $H$ which is not necessarily $K_2$, generalizing edge density conditions for Hamiltonicity and other properties. A \emph{$t$-clique} is a set of $t$ vertices all pairs of which are adjacent. F\"{u}redi, Kostochka, and Luo found the maximum numbers of edges and $t$-cliques in non-$k$-edge-Hamiltonian graphs \cite{FKL19} and the maximum number of copies of $H$ for an arbitrary graph $H$ in sufficiently large non-Hamiltonian graphs in \cite{FKL18}. Luo found the maximum number of $t$-cliques in graphs without long cycles \cite{Luo}. Berikkyzy, Hogenson, Kirsch, and McDonald found the maximum number of $t$-stars in graphs with any of five forbidden properties \cite{BHKM26}. In \cite{DK26} the present authors determined the following $t$-clique density condition for Hamiltonicity as a consequence of \cref{theorem:kr+1ham}. We write $\k(G)$ for the number of $t$-cliques in a graph $G$, and $\G$ for $\T[n-1][r]$ plus a pendant edge to a vertex in a larger part, which was extremal in \cref{theorem:kr+1ham}.

\begin{cor}[Corollary 6.4 in \cite{DK26}]
    Let $G$ be an $n$-vertex, $K_{r+1}$-free, non-Hamiltonian graph where $r \ge 3$. Let $t \ge 2$. If \[n \ge \begin{cases}26 & \text{if }r=3\\ 11 & \text{if } r=4\\ 
2 & \text{if } r \ge 5\end{cases},\quad\text{then } \k(G) \leq \k(\Gs[r][n][]).\]
Equality holds if $G \cong \Gs[r][n][]$, which is $\T[n-1][r]$ plus a pendant edge.
\end{cor}

In \cref{cor:cliques} of the present paper we give an analogous $t$-clique density condition for each of the Hamiltonicity-like properties we consider.

The extremal graphs that maximize both the numbers of edges (Theorems \ref{theorem:kr+1all} and \ref{theorem:kr+1kpath1}) and the numbers of $t$-cliques (\cref{cor:cliques}) are modifications of Tur\'{a}n graphs. Specifically, in each part of \cref{thm:ore} and in the analogous theorems for the other properties (see \cref{thm:kstableedge}), as well as in the clique analogues in \cref{cor:ore}, the extremal graph contains a complete graph $K_{n-1}$ which when replaced by a Tur\'{a}n graph $\T[n-1][r]$ yields the extremal graph for the corresponding $K_{r+1}$-free problem.

\begin{table}[h]
\centering
		\begin{tabular}{ccccc}
			& Unrestricted & Non-Traceable & Non-Hamiltonian & Non-$k$-Hamiltonian\\
			Unrestricted & \tikzstyle{vx}=[inner sep=1.5pt,circle,fill=black,draw=black]
			\tikzstyle{edge}=[thick]
			\begin{tikzpicture}[scale=.6]
				\draw[fill=gray!20] (2,2) circle [radius=1.8];
			\end{tikzpicture} &
            \tikzstyle{vx}=[inner sep=1.5pt,circle,fill=black,draw=black]
			\tikzstyle{edge}=[thick]
			\begin{tikzpicture}[scale=.6]
                \draw[fill=gray!20] (2,2) circle [radius=1.8];
                \foreach \x/\y [count=\i] in {
                4.5/3}
				{
					\node[vx] (V\i) at (\x,\y) {};
				}
			\end{tikzpicture}
            &
            \tikzstyle{vx}=[inner sep=1.5pt,circle,fill=black,draw=black]
			\tikzstyle{edge}=[thick]
			\begin{tikzpicture}[scale=.6]
                \draw[fill=gray!20] (2,2) circle [radius=1.8];
				\foreach \x/\y [count=\i] in {
                3.35/2.65, 4.5/3}
				{
					\node[vx] (V\i) at (\x,\y) {};
				}
				
				\draw[edge] (V1)--(V2);
			\end{tikzpicture}& \tikzstyle{vx}=[inner sep=1.5pt,circle,fill=black,draw=black]
			\tikzstyle{edge}=[thick]
                \begin{tikzpicture}[scale=.6]
                    \draw[fill=gray!20] (2,2) circle [radius=1.8];
                    \draw[dashed][rotate around={45:(3,3)}] (2.15,3) ellipse [x radius=.75, y radius=1.35] node {$N(v)$};
				{
					\node[vx][label=above:$v$] (v) at (5,3) {};
				}
                \draw[edge] (v) -- (1.9,3.5);
                \draw[edge] (v) -- (3.41,1.5);
                
                \draw[rotate around={-25:(3.5,4.55)}] (4.75,3.75) arc
                    [
                    start angle=185,
                    end angle=275,
                    x radius=.65cm,
                    y radius =.65cm
                    ] node[below] {\small{$k+1$}};
                \end{tikzpicture}\\
			&& \cite{OreEdgeCond, ChakrabortiChen} & \cite{OreEdgeCond, ChakrabortiChen}& \cite{CKL70}, \cref{cor:ore}\ref{part:kHamclique}\\
            $K_{r+1}$-free & 
    \tikzstyle{vx}=[inner sep=1.5pt,circle,fill=black,draw=black]
			\tikzstyle{edge}=[thick]
			\begin{tikzpicture}[scale=.6]

             \draw[gray!20, line width=8pt] (3,3) -- (1,3);
                    \draw[gray!20, line width=8pt] (3,3) -- (3,1);
                    \draw[gray!20, line width=8pt] (3,3) -- (1,1);
                    \draw[gray!20, line width=8pt] (1,3) -- (3,1);
                    \draw[gray!20, line width=8pt] (1,3) -- (1,1);
                    \draw[gray!20, line width=8pt] (3,1) -- (1,1);
                    
				\draw[rotate around={-45:(1,3)}, fill=white] (1,3) ellipse [x radius=.5, y radius=.9];
				\draw[rotate around={45:(1,1)}, fill=white] (1,1) ellipse [x radius=.5, y radius=.9];
				\draw[rotate around={45:(3,3)}, fill=white] (3,3) ellipse [x radius=.5, y radius=.9];
				\draw[rotate around={-45:(3,1)}, fill=white] (3,1) ellipse [x radius=.5, y radius=.9];
                
			\end{tikzpicture}  & 
    \tikzstyle{vx}=[inner sep=1.5pt,circle,fill=black,draw=black]
			\tikzstyle{edge}=[thick]
			\begin{tikzpicture}[scale=.6]

             \draw[gray!20, line width=8pt] (3,3) -- (1,3);
                    \draw[gray!20, line width=8pt] (3,3) -- (3,1);
                    \draw[gray!20, line width=8pt] (3,3) -- (1,1);
                    \draw[gray!20, line width=8pt] (1,3) -- (3,1);
                    \draw[gray!20, line width=8pt] (1,3) -- (1,1);
                    \draw[gray!20, line width=8pt] (3,1) -- (1,1);
                    
				\draw[rotate around={-45:(1,3)}, fill=white] (1,3) ellipse [x radius=.5, y radius=.9];
				\draw[rotate around={45:(1,1)}, fill=white] (1,1) ellipse [x radius=.5, y radius=.9];
				\draw[rotate around={45:(3,3)}, fill=white] (3,3) ellipse [x radius=.5, y radius=.9];
				\draw[rotate around={-45:(3,1)}, fill=white] (3,1) ellipse [x radius=.5, y radius=.9];

				\foreach \x/\y [count=\i] in {4.5/3}
				{
					\node[vx] (V\i) at (\x,\y) {};
				}
			\end{tikzpicture} 
            & 
			    \tikzstyle{vx}=[inner sep=1.5pt,circle,fill=black,draw=black]
			\tikzstyle{edge}=[thick]
			\begin{tikzpicture}[scale=.6]

            \draw[gray!20, line width=8pt] (1,1) -- (3,1);
                    \draw[gray!20, line width=8pt] (3,3) -- (1,3);
                    \draw[gray!20, line width=8pt] (3,3) -- (3,1);
                    \draw[gray!20, line width=8pt] (3,3) -- (1,1);
                    \draw[gray!20, line width=8pt] (1,3) -- (3,1);
                    \draw[gray!20, line width=8pt] (1,3) -- (1,1);

				\draw[rotate around={-45:(1,3)}, fill=white] (1,3) ellipse [x radius=.5, y radius=.9];
				\draw[rotate around={45:(1,1)}, fill=white] (1,1) ellipse [x radius=.5, y radius=.9];
				\draw[rotate around={45:(3,3)}, fill=white] (3,3) ellipse [x radius=.5, y radius=.9];
				\draw[rotate around={-45:(3,1)}, fill=white] (3,1) ellipse [x radius=.5, y radius=.9];
				\foreach \x/\y [count=\i] in {
                3.35/2.65, 4.5/3}
				{
					\node[vx] (V\i) at (\x,\y) {};
				}
				
				\draw[edge] (V1)--(V2);
			\end{tikzpicture}
            &

            \tikzstyle{vx}=[inner sep=1.5pt,circle,fill=black,draw=black]
			\tikzstyle{edge}=[thick]
                \tikzstyle{tedge}=[draw=gray, very thick]
			\begin{tikzpicture}[scale=.6]

                    \draw[gray!20, line width=8pt] (1,1) -- (1,3);
                    \draw[gray!20, line width=8pt] (1,1) -- (3,1);
                    \draw[gray!20, line width=8pt] (3,3) -- (1,3);
                    \draw[gray!20, line width=8pt] (3,3) -- (3,1);
                    \draw[gray!20, line width=8pt] (3,3) -- (1,1);
                    \draw[gray!20, line width=8pt] (1,3) -- (3,1);
                    
				\draw[rotate around={-45:(1,3)}, fill=white] (1,3) ellipse [x radius=.5, y radius=.9];
				\draw[rotate around={45:(1,1)}, fill=white] (1,1) ellipse [x radius=.5, y radius=.9];
				\draw[rotate around={45:(3,3)}, fill=white] (3,3) ellipse [x radius=.5, y radius=.9];
				\draw[rotate around={-45:(3,1)}, fill=white] (3,1) ellipse [x radius=.5, y radius=.9];
                    \draw[dashed][rotate around={45:(3,3)}] (2.15,3) ellipse [x radius=.75, y radius=1.35] node {$N(v)$};
				{
					\node[vx][label=above:$v$] (v) at (5,3) {};
				}
                \draw[edge] (v) -- (1.9,3.5);
                \draw[edge] (v) -- (3.41,1.5);
                
                \draw[rotate around={-25:(3.5,4.55)}] (4.75,3.75) arc
                    [
                    start angle=185,
                    end angle=275,
                    x radius=.65cm,
                    y radius =.65cm
                    ] node[below] {\small{$k+1$}};
			\end{tikzpicture}\\
			& \cite{turan, Zykov} & \cref{theorem:kr+1all}\ref{part:trace},  & \cite{DK26}& \cref{theorem:kr+1all}\ref{part:kHam}, \\
            && \cref{cor:cliques}\ref{part:trace}&&\cref{cor:cliques}\ref{part:kHam}\\
		\end{tabular}
		\caption{Extremal graphs maximizing the numbers of edges and $t$-cliques in eight different graph classes. Gray circles represent complete graphs, and gray bands indicate complete multipartite graphs. In the $K_{r+1}$-free row, from left to right, the graphs are $\T[n][4]$, $\Gs[4][n][-1]$, $\Gs[4][n][0]$, and $\Gs[4][n][k]$.
        }\label{fig:graphs1}
	\end{table}

The extremal graphs are pictured in \cref{fig:graphs1}. As in \cite{DK26}, for $-1 \le \ell \le \ceil{(r-1)(n-1)/r}-1$ we write $\Gs[r][n][\ell]$ for the graph formed from $\T[n-1]$ by adding a new vertex $x$ such that the neighborhood of $x$ is isomorphic to $\T[\ell+1][r-1]$ and disjoint from a smallest part of the $\T[n-1]$. In \cref{sec:clique} we discuss the colexicographic order and recall from \cite{DK26} that $\Gs[r][n][\ell]$ is the colex Tur\'{a}n graph $\CT$, where $m = e(\T[n-1][r]) + \ell+1$. Notice that $\G = \Gs[r][n][0]$. The graphs $\Gs[r][n][-1]$, $\Gs[r][n][0]$, $\Gs[r][n][1]$, $\Gs[r][n][k]$, $\Gs[r][n][k]$, $\Gs[r][n][k-2]$, and $\Gs[r][n][k]$ are extremal for \cref{theorem:kr+1all}\ref{part:trace}--\ref{part:kconn} and \cref{theorem:kr+1kpath1}, respectively, and for the corresponding parts of \cref{cor:cliques}.

We give a $t$-clique density condition in $K_{r+1}$-free graphs for an arbitrary stable property that holds for sufficiently large complete graphs in \cref{thm:cliquestability}. If the graph $\Gs$ avoids the $(n+\ell)$-stable property, then it is extremal.

\subsection{Paper outline}

\cref{sec:prelim} provides preliminary definitions and theorems and deduces sufficient edge density and clique density conditions for stable properties in $K_{r+1}$-free graphs in Theorems \ref{thm:stability} and \ref{thm:cliquestability}. In \cref{sec:graphs} we determine which graphs avoid the Hamiltonicity-like properties in the three potentially extremal families of $n$-vertex, $K_{r+1}$-free graphs. \cref{sec:edge} pieces together the theorems of Sections \ref{sec:prelim} and \ref{sec:graphs} to determine the maximum numbers of edges and the extremal graphs which attain them for $r \ge 3$. The main results are Theorems \ref{theorem:kr+1all} and \ref{theorem:kr+1kpath1}. The $r=2$ case behaves differently and is addressed in \cref{sec:r2}. In \cref{sec:clique} we maximize the number of $t$-cliques for all $t \ge 2$ in \cref{cor:cliques}. We conclude with open problems in \cref{sec:open}.

\section{Preliminaries}\label{sec:prelim}

\subsection{Degree and edge conditions for stable properties}\label{subsec:edgedegreestable}

For $s>0$, a property $P$ defined on all $n$-vertex graphs is said to be \emph{$s$-stable} if whenever $G+uv$ has property $P$ (for distinct vertices $u$ and $v$ not adjacent in $G$) and $d_G(u)+d_G(v) \ge s$ then $G$ itself has property $P$. We address all of the following properties in this paper using the fact that they are $s$-stable for some $s=n+\ell$, listed in \cref{table}. 

A graph is \emph{traceable} if it contains a Hamiltonian path, \emph{Hamiltonian} if it contains a Hamiltonian cycle, and \emph{Hamiltonian-connected} if every pair of vertices is connected by a Hamiltonian path. 
An $n$-vertex graph is \emph{$k$-path Hamiltonian} for $0 \le k \le n-2$ if every path of length at most $k$ is contained in a Hamiltonian cycle. 
An $n$-vertex graph is \emph{$k$-Hamiltonian} for $0 \le k \le n-3$ if, for each set $S$ of at most $k$ vertices, $G-S$ is Hamiltonian. 
An $n$-vertex graph is \emph{$k$-Hamiltonian-connected} for $1 \le k \le n-2$ if, for each set $S$ of fewer than $k$ vertices, $G-S$ is Hamiltonian-connected. 
An $n$-vertex graph is \emph{$k$-connected} for $1 \le k \le n-1$ if, for each set $S$ of fewer than $k$ vertices, $G-S$ is connected.

The property of being Hamiltonian is $n$-stable \cite{OreEdgeCond}.
The property of $k$-path Hamiltonicity is $(n+k)$-stable by a proof similar to that of Theorem 1 in \cite{Kronk}. The property of $k$-Hamiltonian-connectedness is $(n+k)$-stable by a proof similar to that of Theorem 3 in \cite{Lick}, using the fact that Hamiltonian-connectedness is $(n+1)$-stable. For the stability of the other properties, see \cite{BC76} (and note that their definition of $k$-Hamilton-connectedness is the one in \cite{Berge}, so the $(n+1)$-stability of Hamiltonian-connectedness is given in \cref{table} under $0$-Hamilton-connectedness).

All of these properties hold for sufficiently large complete graphs. We write $n(P)$ for an integer such that each $K_n$ with $n \ge n(P)$ has property $P$, when such a value exists. 

\begin{table}[h]
\begin{center}
\begin{tabular}{l l l l}
    Property $P$ & $s$ & $\ell$ & $n(P)$\\\hline
    traceability & $n-1$ & $-1$ & $1$\\
    Hamiltonicity & $n$ & $0$ & $3$\\
    Hamiltonian-connectedness & $n+1$ & $1$ & $2$\\
    $k$-path Hamiltonicity & $n+k$ & $k$ & $k+3$\\
    $k$-Hamiltonicity & $n+k$ & $k$ & $k+3$\\
    $k$-Hamiltonian-connectedness & $n+k$ & $k$ & $k+2$\\
    $k$-connectedness & $n+k-2$ & $k-2$ & $k+1$
\end{tabular}
\caption{Each of the listed properties $P$ is $s$-stable for $s=n+\ell$ and holds for all complete graphs $K_n$ with $n \ge n(P)$.}
\label{table}
\end{center}
\end{table}

For all stable properties that hold for sufficiently large complete graphs, Chv\'atal-type degree conditions, edge extremal results, and minimum degree conditions are known.

\subsubsection{Chv\'atal-type degree conditions}

The following theorem can be read out of Bondy and Chv\'atal \cite{BC76}.

\begin{thm}[Bondy and Chv\'atal \cite{BC76}]\label{thm:kstable}
    Let $P$ be an $(n+\ell)$-stable property for which $n(P)$ exists. Let $G$ be a graph with degrees $d_1 \le \cdots \le d_n$ where $n\ge n(P)$. If $G$ does not have $P$, then there is an integer $1 \le i \le (n-1-\ell)/2$ for which $d_i \le i+\ell$ and $d_{n-i-\ell} \le n-i-1$. If $G = K_{i+\ell}\vee(I_i\cup K_{n-2i-\ell})$ does not have $P$, then this condition is best possible.
\end{thm}

While \cref{thm:kstable} applies to $k$-connectedness, the resulting degree sequence condition is not best possible, and Bondy and Boesch gave a stronger theorem with different extremal graphs.

\begin{thm}[Bondy \cite{Bondy69}, Boesch \cite{Boesch74}]\label{thm:sconnconditions}
Let $G$ be a graph with degrees $d_1 \le \cdots \le d_n$ and $1 \le k \le n-1$. If $G$ is not $k$-connected, then for some $1 \le i \le (n-k+1)/2$ we have $d_i \le i+k-2$ and $d_{n-k+1} \le n-i-1$. The graph $G=K_{k-1}\vee(K_i\cup K_{n-k-i+1})$ shows that this condition is best possible.
\end{thm}

\subsubsection{Edge extremal results}\label{subsubsec:edge}
\cref{thm:kstable} implies an edge extremal theorem for each of the same properties.

\begin{thm}\label{thm:kstableedge}
    Let $P$ be an $(n+\ell)$-stable property for which $n(P)$ exists. If $G$ is an $n$-vertex graph that does not have property $P$, and $n \ge n(P)$, then $e(G) \le e(K_{n-1})+\ell+1$. If the graph consisting of $K_{n-1}$ plus a vertex of degree $\ell+1$ does not have property $P$, then the bound is tight. 
\end{thm}

\begin{proof}
    Let $G$ be a graph with degrees $d_1 \le \cdots \le d_n$. Suppose that $G$ does not have property $P$, so by \cref{thm:kstable} there is an integer $1 \le i \le (n-1-\ell)/2$ for which $d_i \le i+\ell$ and $d_{n-i-\ell} \le n-i-1$. Let $x$ and $y$ be vertices of degree $d_{i}$ and $d_{n-i-\ell}$. Notice 
       $d(x)+d(y) \le n+\ell-1$. Then \[e(G) \le d(x) + d(y) + e(G-x-y) \le  n+\ell-1+e(K_{n-2})=e(K_{n-1})+\ell+1.\qedhere\]
\end{proof}

\begin{rem}\label{rem:colexproperties}
For all properties $P$ in \cref{table} and their corresponding values of $\ell$ from the table, the graph consisting of $K_{n-1}$ plus a vertex of degree $\ell+1$ does not have property $P$, so the bound in \cref{thm:kstable} is tight.
\end{rem}

Ore additionally characterized all extremal graphs for traceability, Hamiltonicity, 
and Hamiltonian-connectedness.

\begin{thm}[Ore \cite{OreEdgeCond,OreHConnected}]\label{thm:ore} Let $G$ be a graph on $n$ vertices.
\begin{enumerate}[(a)]
    \item If $G$ is not traceable, then $e(G) \le e(K_{n-1})$. For $n \ge 2$, equality holds if and only if $G$ is a $K_{n-1}$ plus an isolated vertex, with one exceptional graph, $K_{3,1}$, when $n=4$.
    \item If $G$ is not Hamiltonian, then $e(G) \le e(K_{n-1}) +1$. For $n \ge 2$, equality holds if and only if $G$ is a $K_{n-1}$ plus a pendant edge, with one exceptional graph, $K_{3,1,1}$, when $n=5$.
    \item If $G$ is not Hamiltonian-connected, then $e(G) \le  e(K_{n-1}) +2$. For $n \ge 4$, equality holds if and only if $G$ is a $K_{n-1}$ plus a vertex of degree $2$, with one exceptional graph, $K_{3,1,1,1}$, when $n=6$.
\end{enumerate}
\end{thm}

\subsubsection{Minimum degree conditions}

\cref{thm:kstable} also implies the following minimum degree conditions for the same properties, which includes Dirac's minimum degree condition for $P=$ Hamiltonicity \cite{Dirac}.

\begin{thm}\label{thm:stablemindeg}
    Let $P$ be an $(n+\ell)$-stable property for which $n(P)$ exists, and let $G$ be a graph on $n \ge n(P)$ vertices. If $\delta(G) \ge (n+\ell)/2$, then $G$ has property $P$.
\end{thm}

\subsection{Hamiltonicity-like properties in complete multipartite graphs}\label{subsec:CMP}

We use the following three propositions, characterizing which complete multipartite graphs are $k$-Hamiltonian-connected, $k$-Hamiltonian, and traceable, respectively.

\begin{prop}\label{prop:partitekHamconn}
    Let $G$ be a complete multipartite graph on $n$ vertices whose largest part has size $m$. Let $1 \le k \le n-3$. Then $G$ is $k$-Hamiltonian-connected if and only if $m \le (n-k)/2$.
\end{prop}

\begin{proof}
    Suppose $G$ has a part of size greater than $(n-k)/2$. Let $G'$ be the result of deleting some set of $k-1$ vertices from the other parts. Then $G'$ is a complete multipartite graph with one part $P$ of size at least $(n-k+1)/2$ and the other parts having at most $(n-k+1)/2$ vertices. A path of length $\ell$ between two vertices outside $P$ contains at most $\ell/2$ vertices in $P$. A Hamiltonian path between these vertices in $G'$ then is impossible because it would have length $n-k$ but contain more than $(n-k)/2$ vertices in $P$. 
    Since $G'$ is not Hamiltonian-connected, $G$ is not $k$-Hamiltonian-connected.

    Conversely, suppose the largest part of $G$ has size $m \le (n-k)/2$. The degree of a vertex in a part of size $p$ is $n-p$, so the minimum degree of $G$ is $\delta(G)=n-m\ge(n+k)/2$. By \cref{thm:stablemindeg} with $P=k$-Hamiltonian-connectedness and \cref{table}, $G$ is $k$-Hamiltonian-connected.
\end{proof}

\begin{prop}\label{prop:partitekHam}
    Let $G$ be a complete multipartite graph on $n$ vertices whose largest part has size $m$. Let $0 \le k \le n-3$. Then $G$ is $k$-Hamiltonian if and only if $m \le (n-k)/2$.
\end{prop}

\begin{proof}
   Suppose $G$ has a part of size greater than $(n-k)/2$. Let $G'$ be the result of deleting some set of $k$ vertices from the other parts. Then $G'$ is a complete multipartite graph with one part of size at least $(n-k+1)/2$ and the other parts having at most $(n-k-1)/2$ vertices. As any cycle in $G'$ is a cyclic ordering of vertices of $G'$ in which the vertices in the largest part are non-consecutive, we have $G'$ is not Hamiltonian and so, by definition, $G$ is not $k$-Hamiltonian.

   Conversely, suppose the largest part of $G$ has size $m \le (n-k)/2$. The degree of a vertex in a part of size $p$ is $n-p$, so the minimum degree of $G$ is $\delta(G)=n-m\ge(n+k)/2$. This minimum degree condition, by \cref{thm:stablemindeg} with $P=k$-Hamiltonicity and \cref{table}, implies that $G$ is $k$-Hamiltonian.
\end{proof}

\begin{prop}\label{prop:partitetraceable}
    Let $G$ be a complete multipartite graph on $n$ vertices whose largest part has size $m$. Then $G$ is traceable if and only if $m \le (n+1)/2$.
\end{prop}

\begin{proof}
    If $G$ has a part of size greater than $(n+1)/2$, then $G$ is not traceable because any path in $G$ is an ordering of vertices of $G$ in which the vertices in the largest part are non-consecutive.

    Conversely, suppose the largest part of $G$ has size $m \le (n+1)/2$. The degree of a vertex in a part of size $p$ is $n-p$, so the minimum degree of $G$ is $\delta(G) = n-m \ge (n-1)/2$. By \cref{thm:stablemindeg} with $P=$ traceability and \cref{table}, $G$ is traceable.
\end{proof}

The following lemma is used to prove \cref{prop:kpathham} for characterizing the extremal graphs that are not $k$-path Hamiltonian.

\begin{lem}\label{lem:partite_kpath}
    Let $G$ be a complete multipartite graph on $n$ vertices whose largest part has size $m$. If $m > (n-k)/2$ and the remaining $r-1$ parts of $G$ contain a path of length $k$ then $G$ is not $k$-path Hamiltonian.
\end{lem}

\begin{proof}
    Suppose $G$ has a largest part of size greater than $(n-k)/2$. Let $G'$ be the result of deleting the vertices of a path $P$ of length $k$ from the remaining $r-1$ parts of $G$. Then $G'$ is a complete multipartite graph on $n-k-1$ vertices with largest part of size at least $(n-k+1)/2$. 
    By \cref{prop:partitetraceable}, 
    $G'$ is not traceable. Thus there is no Hamiltonian path in $G'$, so 
    no Hamiltonian cycle in $G$ containing $P$. Therefore $G$ is not $k$-path Hamiltonian.
\end{proof}

\subsection{Extremal graphs and their numbers of edges}

In \cite{DK26} we defined two families of potential extremal graphs, $\Gell$ and $\Hl$. Here we recall these definitions and define a third family, $\mathcal{J}^\ell_{n,r}$.

\begin{defn}[$\Gell$, $\Hl$, and $\mathcal{J}^\ell_{n,r}$]\label{def:graphs}
Let $\Gell$ be the set of $n$-vertex, $K_{r+1}$-free graphs consisting of $\T[n-1][r]$ plus a vertex of degree $\ell+1$. Notice that $\G$ is one of the two graphs in $\Gell[r][n][0]$, and $\Gs \in \Gell$ provided that $\ell+1 \le \ceil{(r-1)(n-1)/r}$, so $\Gell$ is nonempty if $n \ge (1+1/(r-1))(\ell+1)+1$.

Let $\Hl$ be the set of $n$-vertex, $K_{r+1}$-free graphs consisting of $\T[(n+1+\ell)/2][r]$ plus an independent set of $(n-1-\ell)/2$ vertices of degree $(n-1+\ell)/2$. Notice that $\Hl$ is empty when $n \equiv \ell \pmod{2}$.

Let $\mathcal{J}^\ell_{n,r}$ be the set of $n$-vertex, $K_{r+1}$-free graphs consisting of $\T[n-1][r]$ plus a vertex of degree $\ell+1$ whose neighborhood is traceable (equivalently, by \cref{prop:partitetraceable}, in the neighborhood of the vertex of degree $\ell + 1$ each part has size at most $(\ell + 2)/2$). Notice that $\mathcal{J}^\ell_{n,r} \subseteq \Gell$, and $\Gs \in \mathcal{J}^\ell_{n,r}$ provided that $\ell+1 \le \ceil{(r-1)(n-1)/r}$, so $\mathcal{J}^\ell_{n,r}$ is nonempty if $n \ge (1+1/(r-1))(\ell+1)+1$.
\end{defn}

Since graphs in $\Hl$ must be $K_{r+1}$-free and contain an $r$-partite Tur\'{a}n graph with neighbors in an independent set, this graph family is empty for large values of $n$ depending on $r$ and $\ell$. The graphs in $\Hl$ have structure as described in the following two lemmas. 

\begin{lem}[Lemma 3.3 in \cite{DK26}]\label{Hfamsize}
    Let $2 \le r \le n-1$ and $-1 \le \ell \le n-3$. If $G \in \Hl$, then $n \le 4r-\ell-3$ and $G$ consists of 
    an independent set $V(J)$ of $(n-1-\ell)/2$ vertices of degree $(n-1+\ell)/2$ and $H \cong \T[(n+1+\ell)/2][r]$ containing at least one partite set of size 1 and all other partite sets of size at most 2. 
    Furthermore, all vertices in partite sets of size $2$ are adjacent to all vertices of $J$.
\end{lem}

\begin{lem}[Lemma 3.4 in \cite{DK26}]\label{Hfamell}
Let $G \in \Hl$ where $\ell \ge -1$ and $n \ge \ell+5$. Then either
\begin{enumerate}[(a)]
    \item \label{part:graph} $G$ is a complete multipartite graph with one partite set of size $(n+1-\ell)/2$ and the other partite sets of sizes at most 2, or
    \item\label{part:deg} for every integer $i$ in $1 \le i \le (n-1-\ell)/2$, either $d_i > i+\ell$ or $d_{n-i-\ell} > n-i-1$.
\end{enumerate}

\end{lem}

Notice that the graphs in $\Gell$ satisfy a Chv\'{a}tal-type condition, 
and we can count their edges based on the numbers of edges in Tur\'{a}n graphs.
\begin{prop}[Proposition 3.2 in \cite{DK26}]\label{prop:gell is extremal}
    For integers $2 \le r \le n-1$ and $-1 \le \ell \le n-3$, suppose $G \in \Gell$. Then $G$ is an $n$-vertex, $K_{r+1}$-free graph. Let $d_1 \le \cdots \le d_n$ be its degrees. Then $d_j \le j+\ell$ for $j=1$, and $e(G)=e(\T[n-1])+\ell+1$.
\end{prop}

While it is convenient to give numbers of edges in the form $e(\T[n])$, we also use bounds on the numbers of edges in Tur\'{a}n graphs and an exact count for small values of $r$.

\begin{prop}[\cite{Lavrov}]\label{prop:turanedgecount} 
    Let $n \ge r \ge 1$. Then $\frac{r-1}{2r}n^2 - \frac{r}{8} \le e(\T) \le \frac{r-1}{2r}n^2$, and, for every $1 \le r \le 7$, $e(\T) = \floor[\big]{\frac{r-1}{2r}n^2}$.
\end{prop}

In comparing bounds involving floor functions, we use the fact that for any real numbers $x$ and $y$, $\floor{x} - \floor{y} > x - 1-y$.

\subsection{New edge and clique density conditions for stable properties}\label{sec:bounds}

In \cite{DK26} we maximized the numbers of edges and $t$-cliques in graphs satisfying a condition on low-degree vertices related to P\'{o}sa's theorem. We also proved that the extremal graphs outside $\Gell$ are quite limited. Here we recall these results (\cref{thm:degcondsummary}, \cref{lem:nrange}, and \cref{cor:posaclique}) and show their significance in combination with \cref{thm:kstable}. They imply new results, Theorems \ref{thm:stability} and \ref{thm:cliquestability}, giving upper bounds on the numbers of edges and $t$-cliques, respectively, in graphs that avoid arbitrary stable properties holding for sufficiently large complete graphs.

The following theorem provides an upper bound on the number of edges in an $n$-vertex, $K_{r+1}$-free graph that satisfies the P\'{o}sa-type condition on low-degree vertices, for all $r \ge 3$.

\begin{thm}[Theorem 4.1 in \cite{DK26}]\label{thm:degcondsummary}
    Let $\ell \ge -1$, $r \ge 3$, and \[n \ge \begin{cases} 6\ell+26 & \text{if } r=3\\ \max\{3+ \ell + \frac{4(\ell+2)}{r-3},5+\ell+\frac{r+2\ell+7}{2r-2}\} & \text{if } 4 \le r \le 7\\ \ell+5+\frac{4(\ell+4)}{r-4} & \text{if } r \ge 8\end{cases}\] be integers, or $(\ell,r,n)=(1,8,10)$. Let $G$ be an $n$-vertex, $K_{r+1}$-free graph with degrees $d_1 \le \cdots \le d_n$. If there is an integer $j$ in $1 \le j \le (n-1-\ell)/2$ such that $d_j \le j+\ell$, then \[ e(G) \leq e(\T[n-1][r]) + (\ell+1). \]
    For $r=3$ and $r \ge 8$, equality holds if and only if $G \in \Gell[r][n][\ell]$. For $4 \le r \le 7$, if $G \in \Gell[r][n]$ then equality holds, and if equality holds then either $G \in \Gell[r][n]$ or $G \in \Hl[r][n]$. In all cases, $\Gell$ is nonempty, and the bound is tight.
\end{thm}

While \cref{thm:degcondsummary} leaves open the possibility of extremal graphs in $\Hl$ for $4 \le r \le 7$, we further proved that these exceptional extremal graphs only occur for small values of $n$.

\begin{lem}[Lemma 5.1 in \cite{DK26}]\label{lem:nrange} Let $\ell \ge -1$, $4 \le r \le 7$, and $n \geq \max\{3+ \ell + \frac{4(\ell+2)}{r-3},5+\ell+\frac{r+2\ell+7}{2r-2}\}$ be integers. Let $G$ be an $n$-vertex, $K_{r+1}$-free graph. Suppose that $G \notin \Gell$ 
and \[ e(G) = e(\T[n-1][r]) + (\ell+1)  = \floor[\Big]{\frac{r-1}{2r}n^2 - \frac{r-1}{r}n +\frac{2\ell r + 3r-1}{2r}}.\]
Then 
\begin{itemize}
    \item if $r=4$, then $5\ell+11 \le n \le 13-\ell$,
    \item if $\ell=-1$ and $5 \le r \le 7$, then $n \le 5$,
    \item if $\ell\ge 0$ and $r=5$, then $3\ell+7\le n \le \min\set{17-\ell,10\ell+10}$, and
    \item if $\ell\ge 0$ and $6 \le r \le 7$, then $n \le \min\{4r-3-\ell,6\ell+6\}$.
\end{itemize}
\end{lem}

Now we combine the above two statements from \cite{DK26} with Bondy and Chv\'{a}tal's sufficient condition for stable properties (\cref{thm:kstable}) to obtain an upper bound on the number of edges in a $K_{r+1}$-free graph avoiding a property $P$.

\begin{thm}\label{thm:stability}
    Let $\ell \ge -1$, and let $P$ be an $(n+\ell)$-stable property for which $n(P)$ exists. Let $n \ge n(P)$, and suppose $n$ is large enough to satisfy the lower bound on $n$ in \cref{thm:degcondsummary} and to exceed the upper bounds on $n$ in \cref{lem:nrange}. 
    Let $r \ge 3$. Let $G$ be an $n$-vertex, $K_{r+1}$-free graph that does not have $P$. Then \[ e(G) \leq e(\T[n-1][r]) + (\ell+1). \]
    The set $\Gell$ is nonempty, and if equality holds then $G \in \Gell$. If at least one graph in $\Gell$ does not have $P$, then the bound is tight.
\end{thm}

\begin{proof}
    By \cref{thm:kstable}, $G$ satisfies the hypothesis of \cref{thm:degcondsummary}. If at least one graph in $\Gell$ does not have $P$, then equality holds for this graph by \cref{prop:gell is extremal}, so the bound is tight. This completes the proof for $r=3$ and $r \ge 8$. For $4 \le r \le 7$ and large enough $n$, by \cref{lem:nrange}, equality implies $G \in \Gell$.
\end{proof}

Notice that for a general stable property $P$ we do not know whether any graphs in $\Gell$ avoid property $P$, so the upper bound on the number of edges given by \cref{thm:stability} may or may not be tight. For example, the property of containing a path $P_k$, for $4 \le k \le n$, is $(n-1)$-stable \cite{BC76}, but for $k < n$, the graphs in $\Gell[r][n][-1]$ contain $P_k$, so \cref{thm:stability} does not determine the maximum number of edges in an $n$-vertex, $P_k$-free, $K_{r+1}$-free graph. Therefore, much of this paper focuses on the properties of \cref{table}, proving that at least some graphs in $\Gell$ (though not necessarily all) avoid each of these properties, and characterizing the extremal graphs.

In \cite{DK26} we also found the maximum number of $t$-cliques in $K_{r+1}$-free graphs satisfying the condition on low-degree vertices, 
though we did not characterize the extremal graphs. 
\begin{cor}[Corollary 6.5 in \cite{DK26}]\label{cor:posaclique}
Let $\ell \ge -1$, $r \ge 3$, and \[n \ge \begin{cases} 6\ell+26 & \text{if } r=3\\ \max\{3+ \ell + \frac{4(\ell+2)}{r-3},5+\ell+\frac{r+2\ell+7}{2r-2}\} & \text{if } 4 \le r \le 7\\ \ell+5+\frac{4(\ell+4)}{r-4} & \text{if } r \ge 8\end{cases}\] be integers, or $(\ell,r,n)=(1,8,10)$. Let $G$ be an $n$-vertex, $K_{r+1}$-free graph with degrees $d_1 \le \cdots \le d_n$. If there is an integer $j$ in $1 \le j \le (n-1-\ell)/2$ such that $d_j \le j+\ell$, then, for all $t \ge 2$, $\k(G) \le \k(\Gs)$. Equality holds if $G \cong \Gs$.
\end{cor}

An immediate consequence is the following $t$-clique density condition for a stable property. 

\begin{thm}\label{thm:cliquestability}
    Let $\ell \ge -1$, and let $P$ be an $(n+\ell)$-stable property for which $n(P)$ exists. Let $n \ge n(P)$, and suppose $n$ also satisfies the lower bound on $n$ in \cref{cor:posaclique}.  
    Let $r \ge 3$. Let $G$ be an $n$-vertex, $K_{r+1}$-free graph that does not have $P$. Then, for all $t \ge 2$, \[ \k(G) \leq \k(\Gs). \]
    If $\Gs$ does not have $P$ then the bound is tight.
\end{thm}

\begin{proof}
    By \cref{thm:kstable}, $G$ satisfies the hypothesis of \cref{cor:posaclique}.
\end{proof}

\section{Hamiltonicity-like Properties in Extremal Families}\label{sec:graphs}

In this section we address the Hamiltonicity-like properties of graphs in $\Gell$, $\Hl$, and $\mathcal{J}^\ell_{n,r}$ in \cref{prop:gell}, \cref{kpathtrace}, and \cref{Hfam}. First we confirm that the graphs in $\Gell$ satisfy the requirements not to have various Hamiltonicity-like properties.

\begin{prop}\label{prop:gell}
    For integers $2 \le r \le n-1$ and $-1 \le \ell \le n-3$, suppose $G \in \Gell$.
    \begin{enumerate}[(a)]
        \item If $\ell = -1$, then $G$ is not traceable and $e(G)=e(\T[n-1])$.
        \item If $\ell = 1$, then $G$ is not Hamiltonian-connected and $e(G)=e(\T[n-1])+2$.
        \item\label{part:lpath} For $\ell \ge 0$, if $G \in \mathcal{J}^\ell_{n,r}$, then $G$ is not $\ell$-path Hamiltonian and $e(G)=e(\T[n-1])+\ell+1$.
        \item For $\ell \ge 0$, $G$ is not $\ell$-Hamiltonian and $e(G)=e(\T[n-1])+\ell+1$.
        \item For $\ell \ge 1$, $G$ is not $\ell$-Hamiltonian-connected and $e(G)=e(\T[n-1])+\ell+1$.
        \item\label{part:gellkconn} $G$ is not $(\ell+2)$-connected and $e(G)=e(\T[n-1])+\ell+1$.
    \end{enumerate}
\end{prop}

\begin{proof}The edge counts are given by \cref{prop:gell is extremal}. Since $G \in \Gell$, $G$ contains a vertex $v$ of degree $\ell+1$.
    \begin{enumerate}[(a)]
        \item If $\ell=-1$, then $G$ contains an isolated vertex so has no Hamiltonian path.
        \item If $\ell=1$, then $G$ contains a vertex $v$ of degree two. Let $x$ and $y$ be the neighbors of $v$. Any path from $x$ to $y$ containing $v$ contains only three vertices, but $n \ge 4$, so there is no Hamiltonian path from $x$ to $y$.
        \item Since $G \in \mathcal{J}^\ell_{n,r}$, the vertex $v$ which has degree $\ell+1$ has a traceable neighborhood. Let $P$ be a path of length $\ell$ in the neighborhood of $v$, and let $u$ and $w$ be the endpoints of $P$. Any cycle containing $v$ and $P$ must contain the edges $\set{u,v}$ and $\set{v,w}$ and therefore contains no other vertices. The length of such a cycle is $\ell+2 \le n-1$, so no Hamiltonian cycle contains $P$.
        \item Deleting any $\ell$ of the neighbors of $v$ yields a graph containing a vertex of degree $1$, which therefore is not Hamiltonian.
        \item Deleting any $\ell-1$ neighbors of $v$ yields a non-Hamiltonian-connected graph. 
        \item Deleting the $\ell+1$ neighbors of $v$ yields a disconnected graph.\qedhere
    \end{enumerate}
\end{proof}

We use \cref{kpathcycle} to characterize the non-$k$-path Hamiltonian graphs in the family $\Gell$ in \cref{kpathtrace}.

\begin{lem}\label{kpathcycle}
    Let $G \in \Gell[r][n][k]$ and $n\ge 8+k+(2k+12)/(r-2)$ where $r \ge 3$ and $k \ge 1$. If $G$ contains a path of length at most $k+4$ that contains the exceptional vertex $v$ of degree $\ell+1$ as an internal vertex then the path is contained in a Hamiltonian cycle.
\end{lem}

\begin{proof}
    Let $G$ be such a graph with path $P = (p_1,\dots,p_i)$ of length $i-1 \le k+4$ containing $v$ as an internal vertex of $P$. Since $n \ge 8+k+(2k+12)/(r-2)$, it follows that \begin{align}
        \lceil (n-1)/r\rceil &\le (n+r-2)/r \le (n-k-6)/2<(n-(k+5))/2.\label{G'partsize}
    \end{align} Let $G' = G - V(P)$, so $G'$ is a complete multipartite graph on $n-i \ge n-(k+5)$ vertices with parts of size at most $\lceil (n-1)/r\rceil< (n-(k+5))/2$ by \cref{G'partsize}. By \cref{prop:partitekHam} with $k=0$, $G'$ is Hamiltonian.

    Let $C'$ be a Hamiltonian cycle in $G'$. As $n \ge 8+k+(2k+12)/(r-2) \ge 8+k$, we have $|V(C')| \ge n-(k+5) \ge 3$. Let $a$, $b$, and $c$ be consecutive vertices in $C'$. As $a$ and $b$ are adjacent in $G' \subset G$, they lie in different parts of $G$. Then $p_{i}$ is adjacent to at least one of $a$ and $b$. Without loss of generality, say $p_i$ is adjacent to $a$. Then concatenating $P$ and $C'$ yields a Hamiltonian path $(p_1, \dots, p_i, a, \dots, b)$ in $G$. 

    If $p_1$ is adjacent to $b$, we have a Hamiltonian cycle $$(p_1, \dots, p_i, a, \dots, b, p_1)$$ in $G$, and so $G$ has a Hamiltonian cycle that contains $P$. If $p_1$ is not adjacent to $b$ then $p_1$ and $b$ lie in the same part, and so $p_1$ is adjacent to $c$.

    By \cref{G'partsize}, every part of $G'$ contains fewer than half the vertices of $G'$, so $C'$ has a pair of consecutive vertices $(u_1, u_2)$ both of which are in parts not containing $b$ so are adjacent to $b$. Replacing $(u_1,u_2)$ with $(u_1,b,u_2)$ in $C'$ and deleting $b$ from $(a,b,c)$ in $C'$ gives the Hamiltonian cycle $$(p_1, \dots, p_i, a, \dots, u_2, b, u_1, \dots, c, p_1)$$ in $G$, and so $G$ has a Hamiltonian cycle that contains $P$.
\end{proof}

\begin{prop}\label{kpathtrace}
    Let $r\ge 3$, $k \ge 0$, and $n \ge \max\set{8+k+(2k+12)/(r-2), k+2r+k/(r-1)}$. If $G \in \Gell[r][n][k]$ is not $k$-path Hamiltonian then $G \in \mathcal{J}^k_{n,r}$. 
\end{prop}

\begin{proof}For $k = 0$, $\mathcal{J}^k_{n,r} = \Gell[r][n][k]$ because the neighborhood of a vertex of degree one is traceable. Now suppose $k \ge 1$.

    Suppose $G$ is not $k$-path Hamiltonian. Then $G$ contains a path $P=(p_1, \ldots, p_i)$ of length $i-1 \le k$ such that $P$ is not contained in a Hamiltonian cycle. We show that if $V(P)$ does not contain the neighborhood $N(v)$ of the exceptional vertex $v$ of degree $\ell+1$, then $P$ is contained in a Hamiltonian cycle. That is, we show that $V(P)$ contains $N(v)$, and since $|V(P)| \le k+1 = |N(v)|$, the path $P$ is a Hamiltonian path of the induced subgraph in the neighborhood of $v$, so $G \in \mathcal{J}^k_{n,r}$.

    We show that $P$ can be extended to a path of length at most $k+4$ that contains $v$ as an internal vertex and so can be extended to a Hamiltonian cycle by \cref{kpathcycle} as $n \ge 8+k+(2k+12)/(r-2)$. If $v \in V(P)$ and $v$ is an internal vertex, we are done. If $v \in V(P)$ and $v$ is an end vertex, say $v=p_1$, there is a $u \in N(v) \setminus V(P)$ such that $(u,v,\dots,p_i)$ is a path of length at most $k+1$ that contains $v$ as an internal vertex. Therefore we assume $v \notin V(P)$. We consider two cases: if at least two vertices of $N(v)$ are not contained in the path $P$ and if exactly one vertex of $N(v)$ is not contained in the path $P$.

    \begin{case} There are distinct vertices $u_1$ and $u_2$ in $N(v)\setminus V(P)$.

    If there is any edge from $u_1$ or $u_2$ to $p_1$ or $p_i$, by concatenating $(u_1, v, u_2)$ and $P$, we have a path of length at most $k+3$ that contains $v$ as an internal vertex. If not, then $u_1, u_2, p_1, p_i$ are all in the same part. Let $X$ be the part containing $\{u_1, u_2, p_1, p_i\}$. As at most $i-1\le k$ vertices in $\set{p_2, \dots, p_{i-1},v}$ are distributed over the $r-1$ parts of $G - X$, there is always a part in $G$ containing at most $\floor{k/(r-1)}$ vertices of $\set{p_2, \dots, p_{i-1},v}$. As $n\ge k+2r+k/(r-1)$ and $r \ge 3$, the number of vertices in this part of $G-X$ which are also not in $\set{p_2, \dots, p_{i-1},v}$ is at least \begin{align}\label{vpath} \lfloor(n-1)/r\rfloor -\lfloor k/(r-1) \rfloor &\ge (n-r)/r - k/(r-1) \ge 1 \end{align} and so there is a vertex $u_3$ in $G-(V(P)\cup\set{u_1,v,u_2})$ that is adjacent to both $u_2$ and $p_1$. The path $(u_1, v, u_2, u_3, p_1, \dots, p_i)$ has length at most $k+4$ and contains $v$ as an internal vertex. 

    \end{case}

    \begin{case} There is exactly one vertex $u \in N(v)\setminus V(P)$, so $N(v)\setminus\set{u}\subseteq V(P)$.
    
    As $P$ has length at most $k$ and $|N(v)|=k+1$, at least one of the end vertices of $P$ (without loss of generality, say $p_1$) is in $N(v)$. Then the path $(u, v, p_1, \dots, p_i)$ of length at most $k+2$ contains $v$ as an internal vertex.\qedhere
    \end{case}
\end{proof}

Thus, by \cref{kpathtrace} and \cref{prop:gell}\ref{part:lpath}, for large enough $n$, the graphs in $\Gell[r][n][k]$ which are not $k$-path Hamiltonian are exactly those in $\mathcal{J}^k_{n,r}$. We now consider the graphs in the family $\Hl[r][n][k]$.

\begin{prop}\label{prop:kpathham}
    Let $G \in \Hl$ where $0 \le \ell \le n-3$ and $r\ge 3$, and suppose $G$ is a complete multipartite graph. Then $G$ is not $\ell$-path Hamiltonian.
\end{prop}
\begin{proof}
    By \cref{Hfamsize}, $n \le 4r-\ell-3$ and $G$ consists of 
    an independent set $V(J)$ of $(n-1-\ell)/2$ vertices of degree $(n-1+\ell)/2$ and $H \cong \T[(n+1+\ell)/2][r]$ containing at least one partite set of size 1 and all other partite sets of size at most 2. 
    Furthermore, all vertices in partite sets of size $2$ are adjacent to all vertices of $J$. 
    
    Let $x, y \in V(J)$, and let $z$ be a vertex in a part of size $1$ in $H$. Since $G$ is complete multipartite, and $x$ and $y$ are nonadjacent, $x$ and $y$ are in the same part of $G$. Then $z$ is either in a different part or in the same part of $G$, so either adjacent or nonadjacent to both $x$ and $y$. Each vertex of $J$ is nonadjacent to one vertex of $H$, and because $G$ is complete multipartite, the vertices of $J$ must all be nonadjacent to the same vertex $z$ of $H$. Then $V(J) \cup \set{z}$ is a partite set of $G$ of size $(n+1-\ell)/2$, and $H-z \cong \T[(n-1+\ell)/2][r-1]$. 
    
    Let $m$ be the size of a largest part in $H-z$. Recall that by \cref{Hfamsize} we have $m \le 2$ and $n \ge 4r-\ell-3$. Since $r \ge 3 \ge m+1$, we have $n \ge 4r-\ell-3\ge 4(m+1)-3-\ell \ge 4m-\ell-1$, and $n \ge 4m-\ell-1$ implies $m \le ((n-1+\ell)/2+1)/2$. Then by \cref{prop:partitetraceable}, $H-z$ is traceable, i.e. has a path on $(n-1+\ell)/2 \ge \ell+1$ vertices; the last inequality follows from $n \ge \ell+3$. 
    By \cref{lem:partite_kpath}, $G$ is not $\ell$-path Hamiltonian. 
\end{proof}

\begin{cor}\label{Hfam}
Let $G \in \Hl[r][n][\ell]$ where $\ell \ge -1$ and $n \ge 5+\ell$. 
\begin{enumerate}[(a)]
    \item For $\ell = -1$, if $G$ is not traceable then $G$ is a complete multipartite graph with one partite set of size $(n+2)/2$ and the other partite sets of sizes at most 2.
    \item For $\ell = 1$, if $G$ is not Hamiltonian-connected then $G$ is a complete multipartite graph with one partite set of size $n/2$ and the other partite sets of sizes at most 2.
    \item For $0 \le \ell \le n-3$, if $G$ is not $\ell$-Hamiltonian then $G$ is a complete multipartite graph with one partite set of size $(n+1-\ell)/2$ and the other partite sets of sizes at most 2. 
    \item For $0 \le \ell \le n-3$, if $G$ is not $\ell$-path Hamiltonian then $G$ is a complete multipartite graph with one partite set of size $(n+1-\ell)/2$ and the other partite sets of sizes at most 2. 
    \item For $1 \le \ell \le n-3$, if $G$ is not $\ell$-Hamiltonian-connected then $G$ is a complete multipartite graph with one partite set of size $(n+1-\ell)/2$ and the other partite sets of sizes at most 2.
    \item $G$ is $(\ell+2)$-connected.
\end{enumerate}
\end{cor}

\begin{proof}
    By \cref{thm:kstable} and \cref{table}, we have the negation of condition \ref{part:deg} in \cref{Hfamell} (with $\ell=-1$ and $\ell=1$ used for the first two parts, respectively). 
    Therefore condition \ref{part:graph} holds. For the last part, by condition \ref{part:graph}, for all $1 \le i < (n-1-\ell)/2$ we have $d_i=(n-1+\ell)/2>i+\ell$, and for $i=(n-1-\ell)/2$, since $n \ge \ell+5$, we have $n-1-\ell > (n+1-\ell)/2$, so $d_{n-1-\ell} \ge n-2 \ge n-i$. By \cref{thm:sconnconditions}, $G$ is $(\ell+2)$-connected.
\end{proof}

\section{Edge Conditions and Extremal Graphs}\label{sec:edge}

Now we can determine all of the extremal graphs outside of $\Gell$ and prove our main edge extremal results.

\subsection{Exceptional extremal graphs}

Generalizing \cite[Lemma 5.3]{DK26}, we determine all extremal graphs outside of $\Gell$. Recall from \cref{thm:degcondsummary} that they occur only when $4 \le r \le 7$ and from \cref{lem:nrange} that they occur only for small values of $n$. We use the notation $\Hlhat$ in Sections \ref{subsec:omnibus} and \ref{subsec:kPH}.

\begin{defn}
    We write $\Hlhat$ for the set of graphs listed in \cref{lem:exceptions}.
\end{defn}

\begin{lem}\label{lem:exceptions}
    If $G \in \Hl$ for $4 \le r \le 7$, $n \ge r+2$, $n \geq \max\{3+ \ell + \frac{4(\ell+2)}{r-3},5+\ell+\frac{r+2\ell+7}{2r-2}\}$, and $e(G) = e(\T[n-1][r])+\ell+1$, then either 
    \begin{itemize}
        \item $\ell=-1$ and $(G,r,n) \in \set{(K_{4,1,1}, 4, 6), (K_{5,1,1,1}, 4, 8)}$,
        \item $\ell=0$ and $(G,r,n) \in \set{(K_{6,2,2,1},4,11), (K_{4,1,1,1},5,7), (K_{5,1,1,1,1},5,9)}$,
        \item $\ell=1$ and $(G,r,n) \in \set{(K_{4,1,1,1,1},6,8), (K_{5,2,1,1,1},5,10), (K_{5,1,1,1,1,1},6,10)}$,
        \item $\ell=2$ and $(G, r, n) \in \set{(K_{6,2,2,2,1}, 5, 13), (K_{5,2,1,1,1,1}, 6, 11), (K_{5,1,1,1,1,1,1}, 7, 11), (K_{4,1,1,1,1,1}, 7, 9)}$,
        \item $\ell=3$ and $(G, r, n) = (K_{5,2,1,1,1,1,1}, 7, 12)$, 
        \item $\ell=4$ and $(G, r, n) \in \set{(K_{6,2,2,2,2,1}, 6, 15), (K_{5,2,2,1,1,1,1}, 7, 13)}$, or
        \item $\ell=6$ and $(G, r, n) = (K_{6,2,2,2,2,2,1}, 7, 17)$.
    \end{itemize} 
    The converse holds too: if $(G,\ell,r,n)$ is in the list above then $G \in \Hl$, $4 \le r \le 7$, $n \ge r+2$, $n \geq \max\{3+ \ell + \frac{4(\ell+2)}{r-3},5+\ell+\frac{r+2\ell+7}{2r-2}\}$, and $e(G) = e(\T[n-1][r])+\ell+1$.
\end{lem}

\begin{proof}As $G \in \Hl$, for $j=(n-1-\ell)/2$, $G$ satisfies the hypotheses of \cref{thm:degcondsummary}, 
so has the maximum number of edges. By \cref{Hfamsize} and \cref{prop:turanedgecount}, we have 
        \begin{equation}\label{eq:Hedges}
            e(G) = \Big\lfloor{\frac{r-1}{2r}n^2-\Big(\frac{r-1}{r}n -\ell\Big)\frac{n-1-\ell}{2}+\frac{3r-1}{2r}\cdot\frac{(n-1-\ell)^2}{4}}\Big\rfloor.
        \end{equation}
    \begin{case} $r=4$
    
        By \cref{lem:nrange}, we have $5\ell+11 \le n \le 13-\ell$, which implies $\ell \in \set{-1,0}$. If $\ell=-1$ then, by \cref{lem:nrange}, we have $6 \le n \le 14$, so $n \in \{6,8,10,12,14\}$ (because $\Hl$ is empty if $n\equiv\ell \pmod{2}$). 
        If  $\ell=-1$ and $n \ge 10$, by  \cref{eq:Hedges}, we have \begin{align*}
        e(\T[n-1][4])-e(G)&= \Big\lfloor{\frac{3}{8}(n-1)^2}\Big\rfloor - \Big\lfloor{\frac{3}{8}n^2-\left(\frac{3n}{4}+1\right)\frac{n}{2}+\frac{11n^2}{32}}\Big\rfloor\\
        &> \frac{3n^2}{8}-\frac{3n}{4}+\frac{3}{8}-1 -\left(\frac{11n^2}{32}-\frac{n}{2}\right)
        =\frac{n^2}{32}-\frac{n}{4}-\frac{5}{8}\ge 0,
    \end{align*} contradicting $e(G) = e(\T[n-1][r])$. If $\ell=0$ then $n \in \set{11,13}$. For each of the pairs $(\ell,n)$ identified above, we use \cref{eq:Hedges} to find $e(G)$ and compare it to $e(\T[n-1][4])+\ell+1$. 
    When these numbers are equal, we find the unique extremal graph $G \in \Hl[4][n][\ell]$ using \cref{Hfam} and the extremality of $G$. (In fact for each pair $(\ell,n)$ we can find a unique potential extremal graph $G \in \Hl[4][n][\ell]$ using \cref{Hfam} and the extremality of $G$, but we include this graph in the table below only when it is extremal.) 

        \begin{center}
        \begin{tabular}{lllll}
        $(\ell,n)$ & $e(G)$ & $e(\T[n-1][4])+\ell+1$ & Extremal graph $G \in \Hl[4][n][\ell]$\\\hline
        $(-1,6)$ & 9 & 9 & $K_{4,1,1}$\\
        $(-1,8)$ & 18 & 18 & $K_{5,1,1,1}$ \\
        $(0,11)$ & 38 & 38 & $K_{6,2,2,1}$\\
        $(0,13)$ & 54 & 55 & -
        \end{tabular}
        \end{center}
    \end{case}

    \begin{case} $r=5$

        By \cref{lem:nrange}, if $\ell=-1$, then we have $n \le 5$, but we have assumed that $n \ge r+2 = 7$. Otherwise $\ell \ge 0$, and we have $3\ell+7 \le n \le \min\set{17-\ell, 10\ell+10}$, so $0 \le \ell \le 2$ where $n \not\equiv \ell \pmod 2$. In each case, we 
        check the number of edges in $G$ by \cref{eq:Hedges} and $e(\T[n-1][5])+\ell+1$, and when they are equal determine a unique graph using the extremality of $G$ and \cref{Hfam}. The results are summarized in the following table.

        \begin{center}
        \begin{tabular}{lllll}
        $(\ell,n)$ & $e(G)$ & $e(\T[n-1][5])+\ell+1$ & Extremal graph $G \in \Hl[5][n][\ell]$\\\hline
        $(0,7)$   & 15 & 15 & $K_{4,1,1,1}$\\
        $(0,9)$   & 26 & 26 & $K_{5,1,1,1,1}$\\
        $(1,10)$   & 34 & 34 & $K_{5,2,1,1,1}$\\
        $(1,12)$   & 49 & 50 & -\\
        $(1,14)$   & 67 & 69 & -\\
        $(1,16)$   & 88 & 92 & -\\
        $(2,13)$   & 60 & 60 & $K_{6,2,2,2,1}$\\
        $(2,15)$   & 80 & 81 & -
        \end{tabular}
        \end{center}
        
    \end{case}

    \begin{case} $r=6$
    
        First note $n\ge r+2 =8$. By \cref{lem:nrange}, if $\ell=-1$, we have $n \le 5$, a contradiction. Otherwise we have $\ell \ge 0$ and $\max\set{(7\ell+17)/3,(12\ell+63)/10} \le n \le \min\set{21-\ell,6\ell+6}$. From $(7\ell+17)/3 \le 21-\ell$ we have $\ell\le 4$. If $\ell=0$, we have $n \le 6\ell+6=6$, a contradiction. 
        Thus $1 \le \ell \le 4$.

        By substituting $r=6$ into \cref{eq:Hedges}, and using the expression for $e(\T[n-1][r])+\ell+1$ from \cref{thm:degcondsummary}, we have
        \begin{align}\label{eq:twobounds}
            e(G) &= \Big\lfloor{\frac{5}{12}n^2-\Big(\frac{5}{6}n -\ell\Big)\frac{n-1-\ell}{2}+\frac{17}{12}\cdot\frac{(n-1-\ell)^2}{4}}\Big\rfloor\nonumber\\
            &=  \floor{(17n^2 + (10\ell-14)n+(-7\ell^2+10\ell+17))/48}\nonumber\\
            &= \floor[\Big]{\frac{5}{12}n^2 - \frac{5}{6}n+\frac{12\ell+17}{12}} =  \floor{(20n^2-40n+48\ell+68)/48},
        \end{align}
        and for each of the values $\ell \in \set{1,2,3,4}$ one can find a minimum value of $n$ for which $(17n^2 + (10\ell-14)n+(-7\ell^2+10\ell+17))/48 + 1 \le (20n^2-40n+48\ell+68)/48$, contradicting \cref{eq:twobounds}. In this way we eliminate the cases where $\ell=1$ and $n \ge 11$, $\ell=2$ and $n\ge 13$, $\ell=3$ and $n \ge 15$, and $\ell=4$ and $n \ge 17$. 
        Also using the facts that $n \ge \max\set{(7\ell+17)/3,(12\ell+63)/10}$ and $n \not\equiv \ell \pmod{2}$, the remaining cases are $(\ell,n) \in \set{(1,8), (1,10), (2,11), (3,14), (4,15)}$. Similarly to the $r=4$ and $r=5$ cases, we determine $e(G)$ by \cref{eq:Hedges}, $e(\T[n-1][6])+\ell+1$, and where applicable the unique extremal graph in $\Hl[6][n][\ell]$ using the extremality of $G$ and \cref{Hfam}.

        \begin{center}
        \begin{tabular}{lllll}
        $(\ell,n)$ & $e(G)$ & $e(\T[n-1][6])+\ell+1$ & Extremal graph $G \in \Hl[6][n][\ell]$\\\hline
        $(1,8)$   & 22 & 22 & $K_{4,1,1,1,1}$\\
        $(1,10)$   & 35 & 35& $K_{5,1,1,1,1,1}$\\
        $(2,11)$   & 44 & 44& $K_{5,2,1,1,1,1}$\\
        $(3,14)$   & 73 & 74& -\\
        $(4,15)$   & 86 & 86& $K_{6,2,2,2,2,1}$
        \end{tabular}
        \end{center}
    \end{case}

    \begin{case} $r=7$

        First note $n \ge r+2 = 9$. By \cref{lem:nrange}, if $\ell=-1$, we have $n \le 5$, a contradiction. Otherwise we have $\ell \ge 0$ and $\max\{2\ell+5,\frac{7\ell+37}{6}\} \le n \le \min\{25-\ell,6\ell+6\}$. From $2\ell+6 \le n \le 25-\ell$ we have $\ell \le 6$. If $\ell=0$, then $n \le 6\ell+6=6$, a contradiction. 
        Thus $1 \le \ell \le 6$.

        By substituting $r=7$ into \cref{eq:Hedges}, and using the expression for $e(\T[n-1][r])+\ell+1$ from \cref{thm:degcondsummary}, we have
        \begin{align}\label{eq:twobounds7}
            e(G) &= \Big\lfloor{\frac{6}{14}n^2-\Big(\frac{6}{7}n -\ell\Big)\frac{n-1-\ell}{2}+\frac{20}{14}\cdot\frac{(n-1-\ell)^2}{4}}\Big\rfloor\nonumber\\
            &=  \floor{(5n^2 + (3\ell-4)n+(-2\ell^2+3\ell+5))/14}\nonumber\\
            &= \floor[\Big]{\frac{6}{14}n^2 - \frac{6}{7}n+\frac{14\ell+20}{14}} =  \floor{(6n^2-12n+14\ell+20)/14},
        \end{align}
        and for each of the values $\ell \in \set{1,2,3,4,5,6}$ one can find a minimum value of $n$ for which $(5n^2 + (3\ell-4)n+(-2\ell^2+3\ell+5))/14 + 1 \le (6n^2-12n+14\ell+20)/14$, contradicting \cref{eq:twobounds7}. Thus we eliminate the cases where $\ell=1$ and $n \ge 10$, $\ell=2$ and $n\ge 12$, $\ell=3$ and $n \ge 13$, $\ell=4$ and $n \ge 15$, $\ell=5$ and $n\ge 17$, and $\ell=6$ and $n \ge 19$. Also using $n\ge 2\ell+5$ and $n \not\equiv \ell \pmod{2}$, the remaining cases are the pairs $(\ell,n)$ shown in the leftmost column table in the below.

        \begin{center}
        \begin{tabular}{lllll}
        $(\ell,n)$ & $e(G)$ & $e(\T[n-1][7])+\ell+1$&Extremal graph $G \in \Hl[7][n][\ell]$\\\hline
        $(1,8)$   & 22 & 23 & - \\
        $(2,9)$   & 30 & 30 & $K_{4,1,1,1,1,1}$\\
        $(2,11)$ & 45 & 45 & $K_{5,1,1,1,1,1,1}$ \\
        $(3,12)$ & 55 & 55 & $K_{5,2,1,1,1,1,1}$ \\
        $(4,13)$  & 66 & 66 & $K_{5,2,2,1,1,1,1}$\\
        $(5,16)$ & 101 & 102 & - \\
        $(6,17)$ & 116 & 116 & $K_{6,2,2,2,2,2,1}$ \qedhere
        \end{tabular}
        \end{center}
    \end{case}
\end{proof}

    \subsection{An omnibus theorem}\label{subsec:omnibus}

The following theorem gives the maximum number of edges in $n$-vertex, $K_{r+1}$-free graphs that avoid one of a list of forbidden properties and characterizes the extremal graphs. The known result for Hamiltonicity (\cref{theorem:kr+1ham}) is included as part \ref{part:Ham} in the statement of the theorem for completeness.

\begin{thm}\label{theorem:kr+1all}
Let $G$ be an $n$-vertex, $K_{r+1}$-free graph where $r\ge 3$.

\begin{enumerate}[(a)]
\item\label{part:trace} If $G$ is not traceable and
\[n \ge \begin{cases}20 & \text{if }r=3\\ 
1 & \text{if } r \ge 4
\end{cases},\quad\text{ then } e(G) \leq e(\T[n-1][r]).\] Equality holds if and only if $G$ is $\T[n-1][r]$ plus an isolated vertex, i.e., $G \in \Gell[r][n][-1]$, or $G$ is $K_{3,1}$, $K_{4,1,1}$, or $K_{5,1,1,1}$, 
with the exceptional graphs occurring in the cases where $r \ge 4$ and $n=4$, or $(r,n)$ is $(4,6)$ or $(4,8)$, respectively.

\item \label{part:Ham}If $G$ is not Hamiltonian and \[n \ge \begin{cases}26 & \text{if }r=3\\ 11 & \text{if } r=4\\
2 & \text{if } r \ge 5\end{cases},\quad\text{ then }e(G) \leq e(\T[n-1][r]) + 1.\] Equality holds if and only if $G \in \Gell[r][n][0]$ or $G$ is $K_{3,1,1}$, $K_{6,2,2,1}$, $K_{4,1,1,1}$, or $K_{5,1,1,1,1}$, with the exceptional graphs occurring in the cases where $r \ge 5$ and $n =5$, or $(r,n)$ is $(4,11)$, $(5,7)$, or $(5,9)$, respectively.

\item If $G$ is not Hamiltonian-connected and \[n \ge \begin{cases} 32 & \text{if } r=3\\ 16 & \text{if } r=4\\ 10 & \text{if } r=5\\
4 & \text{if } r \ge 6\end{cases},\quad\text{ then }e(G) \leq e(\T[n-1][r]) + 2.\]  Equality holds if and only if $G \in \Gell[r][n][1]$ or $G$ is $K_{3,1,1,1}$, $K_{4,1,1,1,1}$, $K_{5,2,1,1,1}$, or $K_{5,1,1,1,1,1}$ with the exceptional graphs occurring in the cases where $r \ge 6$ and $n=6$, or $(r,n)$ is $(6,8)$, $(5,10)$ or $(6,10)$, respectively.

\item\label{part:kHam} For every $k \ge 0$, if $G$ is not $k$-Hamiltonian and \[n \ge \begin{cases} 6k+26 & \text{if } r=3\\ 5k+11 & \text{if } 4 \le r \le 7\\ k+5+4(k+4)/(r-4) & \text{if } r \ge 8\end{cases},\quad\text{ then }
 e(G) \leq e(\T[n-1][r]) + (k+1). \] 
Equality holds if and only if $G \in \Gell[r][n][k]$ or $G \in \Hlhat[r][n][k]$.

\item For every $k \ge 1$, if $G$ is not $k$-Hamiltonian-connected and \[n \ge \begin{cases} 6k+26 & \text{if } r=3\\ 5k+11 & \text{if } r=4\\ 3k+7 & \text{if } 5 \le r \le 7\\ k+5+4(k+4)/(r-4) & \text{if } r \ge 8\end{cases},\quad\text{ then }
 e(G) \leq e(\T[n-1][r]) + (k+1). \] 
Equality holds if and only if $G \in \Gell[r][n][k]$ or $G \in \Hlhat[r][n][k]$.

\item\label{part:kconn} For every $k \ge 1$, if $G$ is not $k$-connected and \[n \ge \begin{cases} 6k+14 & \text{if } r=3\\ 5k+1 & \text{if } 4 \le r \le 7\\ 2k+5 & \text{if } r \ge 8\end{cases},\quad\text{ then }
 e(G) \leq e(\T[n-1][r]) + (k-1). \] 
Equality holds if and only if $G \in \Gell[r][n][k-2]$.
\end{enumerate}
In all cases 
$\Gell$ is nonempty, and the bounds are tight.
\end{thm}

\begin{proof}
    The proof of each part (a) and (c)--(f) involves selecting an appropriate value of $\ell$ depending on the part, as shown in the following table (and corresponding to the forbidden property being $(n+\ell)$-stable), and then substituting that value for $\ell$ throughout the rest of the proof below. 
    \begin{center}    
    \begin{tabular}{|c|c|c|c|c|c|}\hline
        Part & (a) & (c) & (d) & (e) & (f) \\\hline
        $\ell$ & $-1$ & $1$ & $k$ & $k$ & $k-2$\\\hline
    \end{tabular}
    \end{center}

    Using the appropriate value of $\ell$ in each part, equality is achieved for all properties if $G \in \Gell$ by \cref{prop:gell}.

    In each part, for the appropriate value of $\ell$, there is an integer $j$ in $1 \le j \le (n-1-\ell)/2$ such that $d_j \le j+\ell$, using \cref{thm:kstable} and \cref{table}.

    Notice that for all properties, when $r=3$, we have assumed that $n \ge 6\ell+26$. Therefore \cref{thm:degcondsummary} implies that $e(G) \le e(\T[n-1][r]) + (\ell+1)$, that $\Gell$ is nonempty, and that equality holds if and only if $G \in \Gell[r][n][\ell]$ for $r=3$.
    
    For parts (a) and (c) of the theorem (corresponding to $\ell \in \set{-1, 1}$), we address the $n-1 \le r$ case separately. If $n-1 \le r$ then the graph consisting of $K_{n-1}$ and a vertex of degree $\ell+1$ both is $K_{r+1}$-free and has the maximum number of edges among all $n$-vertex graphs avoiding the forbidden property by \cref{thm:ore}. Therefore this graph also is extremal among such graphs that are $K_{r+1}$-free, and $e(G) \le e(K_{n-1})+(\ell+1) = e(\T[n-1][r])+(\ell+1)$, with equality if and only if $G$ is a $K_{n-1} = \T[n-1][r]$ plus a vertex of degree $\ell+1$, with one exceptional graph for each property: $K_{3,1}$ and $K_{3,1,1,1}$, respectively. Now we assume $n \ge r+2$ for parts (a) and (c).
    
    When $r\ge 8$, notice that for all properties we have assumed $n \ge \ell+5 + 4(\ell+4)/(r-4)$ (using the previous paragraph for properties  (a) and (c) to assume $n \ge r+2 \ge 10$, and for $k$-connectedness using $\ell+5+4(\ell+4)/(r-4) \le 2\ell+9 = 2(k-2)+9$) with one exception, $(\ell,r,n)=(1,8,10)$. In all cases \cref{thm:degcondsummary} 
    implies that $e(G) \le e(\T[n-1][r])+(\ell+1)$, that $\Gell$ is nonempty, and that equality holds if and only if $G \in \Gell$ for all $r \ge 8$.

    For $4 \le r \le 7$ we note:
    \begin{enumerate}[(a)]
        \item For traceability, $n \ge r+2 \ge 6 \ge \max \{2+\frac{4}{r-3}, 4+\frac{r+5}{2r-2}\} = \max\{3+ \ell + \frac{4(\ell+2)}{r-3},5+\ell+\frac{r+2\ell+7}{2r-2}\}$.
        \item[(c)] For Hamiltonian-connectedness, since $n \ge r+2$, 
            \[
            n \ge \begin{cases} 16 & \text{if } r=4\\
            10 & \text{if } r=5\\
            8 & \text{if } 6\le r \le 7\\\end{cases} = 
            \ceil[\Big]{\max\{3+ 1 + \frac{4(1+2)}{r-3},5+1+\frac{r+2+7}{2r-2}\}}.
            \]
    \end{enumerate}
    For parts (d)--(e) of the theorem we use $n \ge 5k+11 \ge \max\{3+ k + \frac{4(k+2)}{r-3},5+k+\frac{r+2k+7}{2r-2}\}$ and $5k+11 \ge 11 \ge r+2$. For part (d) since $k\ge 1$ we additionally use that for $5 \le r \le 7$, $n \ge 3k+7 \ge 10 \ge \max\{r+2,3+ k + \frac{4(k+2)}{r-3},5+k+\frac{r+2k+7}{2r-2}\}$. 
    
    Therefore for $4 \le r \le 7$ and parts (a) and (c)--(f), \cref{thm:degcondsummary} with the appropriate value of $\ell$ implies that $e(G) \le e(\T[n-1][r])+(\ell+1)$, that $\Gell$ is nonempty, that if $G \in \Gell$ then equality holds, and that if we have equality then $G \in \Gell[r][n][\ell]$ or $G \in \Hl[r][n][\ell]$. By \cref{lem:exceptions}, in the latter case we have $G \in \Hlhat[r][n][\ell]$. By Propositions \ref{prop:partitetraceable} (part (a)), \ref{prop:partitekHam} (part (d)), and \ref{prop:partitekHamconn} (parts (c) and (e)), the graphs in $\Hlhat$ avoid the forbidden properties as required. However, for part (f), recall from \cref{Hfam} that the graphs in $\Hl[r][n][k-2]$ are $k$-connected, so there are no extremal graphs in this family.
\end{proof}

\subsection{$k$-Path Hamiltonicity}\label{subsec:kPH}

In this subsection we prove two theorems, the first proving the best possible bound on the number of edges in a $K_{r+1}$-free graph that is not $k$-path Hamiltonian for $n$ sufficiently large. In the second theorem, with more restrictive lower bounds on $n$, we give a complete description of all extremal graphs.

\begin{thm}\label{theorem:kr+1kpath1}
Let $G$ be an $n$-vertex, $K_{r+1}$-free graph where $r \ge 3$. For every $k \ge 0$, if $G$ is not $k$-path Hamiltonian and \[n \ge \begin{cases} 6k+26 & \text{if } r=3\\ 5k+11 & \text{if } 4 \le r \le 7\\  k+5+ 4(k+4)/(r-4) & \text{if } r \ge 8\end{cases},\quad\text{then }e(G) \leq e(\T[n-1][r]) + (k+1).\] 
The set $\mathcal{J}^k_{n,r}$ is nonempty, and equality holds if $G \in \mathcal{J}^k_{n,r}$.
\end{thm}

\begin{proof}
If $G \in \mathcal{J}^k_{n,r}$ then $G$ is not $k$-path Hamiltonian and $e(G) = e(\T[n-1][r]) + (k+1)$ by \cref{prop:gell} as $\mathcal{J}^k_{n,r} \subseteq \Gell[r][n][k]$. Otherwise the proof follows similarly to that of \cref{theorem:kr+1all}, using $\ell=k$, \cref{thm:kstable}, and \cref{table} to establish the existence of $j$ such that $d_j \le j+\ell$.
\end{proof}

With greater lower bounds on $n$ we use \cref{kpathtrace} to characterize the extremal graphs.

\begin{thm}\label{theorem:kr+1kpath2}
Let $G$ be an $n$-vertex, $K_{r+1}$-free graph where $r \ge 3$. For every $k \ge 0$, if $G$ is not $k$-path Hamiltonian and \[n \ge \begin{cases} 6k+26 & \text{if } r=3\\ 6k+2r+3 & \text{if } 4 \le r \le 7\\ k+4(k+4)/(r-4)+2r & \text{if } r \ge 8\end{cases},\quad\text{ then } e(G) \leq e(\T[n-1][r]) + (k+1). \] 
The set $\mathcal{J}^k_{n,r}$ is nonempty, and equality holds if and only if $G \in \mathcal{J}^k_{n,r}$ or $G \in \Hlhat[r][n][k]$.
\end{thm}

\begin{proof}
The upper bound $e(G) \leq e(\T[n-1][r]) + (k+1)$ holds, with equality if $G \in \mathcal{J}^k_{n,r}$, by \cref{theorem:kr+1kpath1}, and equality holds if $G \in \Hlhat[r][n][k]$ by the latter statement of \cref{lem:exceptions} (though some graphs in $\Hlhat$ have fewer vertices than required in this theorem).

For $r \ge 8$ and $n \ge  k+4(k+4)/(r-4)+2r >  k+5+ 4(k+4)/(r-4)$, \cref{thm:degcondsummary} 
with $\ell=k$ implies that equality holds if and only if $G \in \Gell[r][n][k]$. For $r=3$ and $n \ge 6k+26$, \cref{thm:degcondsummary} with $\ell=k$ implies that equality holds if and only if $G \in \Gell[r][n][k]$.

For $4 \le r \le 7$ and $n \ge 6k+2r+3 \ge \max\{3+ k + \frac{4(k+2)}{r-3},5+k+\frac{r+2k+7}{2r-2},r+2\}$, \cref{thm:degcondsummary} 
with $\ell=k$ implies that if we have equality then $G \in \Gell[r][n][k]$ or $G \in \Hl[r][n][k]$. In the latter case, \cref{lem:exceptions} implies $G \in \Hlhat$. By \cref{prop:kpathham}, these graphs are not $k$-path Hamiltonian.

If $G \in \mathcal{J}^k_{n,r}$, then $G$ is not $k$-path Hamiltonian by \cref{prop:gell}\ref{part:lpath}. It remains to show that if $G \in \Gell[r][n][k]$ is not $k$-path Hamiltonian then $G \in \mathcal{J}^k_{n,r}$. The case $(k,r)=(0,4)$ is included in \cref{theorem:kr+1all}\ref{part:Ham} as 0-path Hamiltonicity is Hamiltonicity and $\Gell[r][n][0] = \mathcal{J}^0_{n,r}$. In all other cases, this implication holds by \cref{kpathtrace}; it can be checked that the lower bound on $n$ in \cref{theorem:kr+1kpath2} exceeds the lower bounds in \cref{kpathtrace}.
\end{proof}

\begin{rem} All of the graphs in $\Hlhat$ for $0 \le \ell \le n-3$ are known to be extremal (and not $\ell$-path Hamiltonian) by \cref{prop:kpathham}, \cref{thm:degcondsummary}, 
and \cref{lem:exceptions}, although some of them have smaller values of $n$ than are included in \cref{theorem:kr+1kpath2}. 
\end{rem}

\section{Edge Density Conditions for $r=2$}\label{sec:r2}

We have seen that, for $r \ge 3$ and sufficiently large $n$, each extremal graph consists of the Tur\'an graph on $n-1$ vertices (a balanced complete $r$-partite graph) plus one low-degree vertex. In contrast, in the case that $r=2$, the extremal graphs are complete bipartite graphs that are balanced or almost balanced, except for the property of $k$-connectedness. (In bipartite graphs, unlike $r$-partite graphs for $r \ge 3$, paths and cycles must alternate between partite sets.) Our proofs use Mantel's theorem.

\begin{thm}[Mantel \cite{Mantel}]
    For every $n \ge 1$, if $G$ is an $n$-vertex, $K_3$-free graph, then $e(G) \le e(\T[n][3])$, and equality holds if and only if $G \cong \T[n][3]$. 
\end{thm}

\begin{thm}
    For $n \ge 6$, let $G$ be an $n$-vertex, $K_3$-free graph that is not traceable. Then \[e(G) \le e(K_{\floor{n/2}-1,\ceil{n/2}+1}) = \begin{cases}n^2/4-1&\text{if $n$ is even}\\
    (n^2-9)/4&\text{if $n$ is odd}\end{cases},\]
    and this bound is tight as $K_{\floor{n/2}-1,\ceil{n/2}+1}$ is an $n$-vertex, $K_3$-free graph that is not traceable. For $n \ge 8$, $K_{\floor{n/2}-1,\ceil{n/2}+1}$ is the unique extremal graph.
\end{thm}
\begin{proof}
    Let $G$ be such a graph. By Mantel's theorem, every $n$-vertex, $K_3$-free graph $G$ has $e(G) \le \floor{n^2/4}$, with equality if and only if $G \cong K_{\floor{n/2},\ceil{n/2}}$, which is traceable. As $G$ is not traceable, we have $e(G) \le \floor{n^2/4}-1$. For even $n$, the graph $K_{\floor{n/2}-1,\ceil{n/2}+1} = K_{n/2-1,n/2+1}$ is not traceable because it is bipartite with part sizes differing by two. This graph has $e(K_{\floor{n/2}-1,\ceil{n/2}+1}) = n^2/4-1$ edges, completing the proof of the tight upper bound for every even $n \ge 2$.
    
    In the remainder of the proof, we prove the upper bound for odd $n \ge 7$ and the uniqueness of the extremal graphs for all $n\ge 8$. We consider three cases: $G$ is not bipartite, $G$ is a subgraph of $K_{\floor{n/2},\ceil{n/2}}$, and $G$ is bipartite with part sizes differing by more than one.
    
    \begin{case} $G$ is not bipartite

    Let $f(n)$ be the maximum number of edges in an $n$-vertex, non-bipartite, $K_3$-free graph. Note that 
    \[
    e(K_{\floor{n/2}-1,\ceil{n/2}+1}) = 
    \begin{cases}
    (n^2-9)/4 = (n-1)^2/4+1   &\text{for } n=7\\
    (n^2-9)/4 > (n-1)^2/4+1  &\text{for odd } n \ge 9\\
    n^2/4-1 > (n-1)^2/4+1 &\text{for even } n \ge 6,\\
    \end{cases}
    \]
    so it is enough to show that $f(n) \le (n-1)^2/4+1$, which is Exercise 1.1.5 in \cite{Zhao23}.
    \end{case}
    
    \begin{case} $G$ is a subgraph of $K_{\floor{n/2},\ceil{n/2}}$

    If $n$ is even, then $K_{\floor{n/2},\ceil{n/2}}$ is Hamiltonian. Deleting any edge of a Hamiltonian graph yields a traceable graph. Therefore $G$ is obtained by deleting at least two edges of $K_{\floor{n/2},\ceil{n/2}}$, so $e(G) \le e(K_{\floor{n/2},\ceil{n/2}}) - 2 < n^2/4 - 1$, and no extremal graphs exist in this case with even $n$.

Now suppose $n$ is odd. The graph $K_{\floor{n/2},\ceil{n/2}}=K_{(n-1)/2,(n+1)/2}$ is traceable, so $G$ is a proper subgraph of $K_{(n-1)/2,(n+1)/2}$. We show that for $n \ge 5$ deleting any edge from $K_{(n-1)/2,(n+1)/2}$ yields a traceable graph, so it is necessary to delete at least two edges to obtain a non-traceable graph, and $e(G) \le e(K_{(n-1)/2,(n+1)/2}) - 2 = (n^2-9)/4$. Then we show that for $n \ge 7$ this inequality is strict, so no extremal graphs exist in this case.

Let $G$ be the graph obtained by deleting one edge from $K_{(n-1)/2,(n+1)/2}$, which is independent of the choice of edge. Without loss of generality, label the endpoints of the deleted edge with $1$ and $n-1$, where $1$ is the vertex in the larger part, and $n-1$ is the vertex in the smaller part. Label the other vertices in the larger part with the odd numbers from $3$ to $n$ and the other vertices in the smaller part with the even numbers from $2$ to $n-3$. Then $(1,2,\ldots, n-1, n)$ is a Hamiltonian path of $G$, using the fact that $n \ge 4$ (so the pairs $\set{1,2}$ and $\set{n-2,n-1}$ are edges and not $\set{1,n-1}$). Thus $G$ is traceable.

For $n \ge 7$, we show that deleting any two edges from $K_{(n-1)/2,(n+1)/2}$ yields a traceable graph, so $e(G) \le e(K_{(n-1)/2,(n+1)/2}) - 3 < (n^2-9)/4$. If the two deleted edges are not incident, then without loss of generality label them $\set{1, n-3}$ and $\set{3,n-1}$, where $1$ and $3$ are in the larger part and $n-3$ and $n-1$ are in the smaller part. Label the other vertices in the larger part with the odd numbers from $5$ to $n$ and the other vertices in the smaller part with the even numbers from $2$ to $n-5$. Using the fact that $n \ge 6$, the pairs $\set{1,2}$, $\set{n-4,n-3}$, $\set{3,4}$, and $\set{n-2,n-1}$ all are not $\set{1,n-3}$ or $\set{3,n-1}$ so are edges, so $(1,2,\ldots,n-1,n)$ is a Hamiltonian path of $G$. 

Otherwise, the two deleted edges are incident. If the deleted edges are incident at a vertex in the larger part, then without loss of generality label the deleted edges $\set{1,n-3}$ and $\set{1,n-1}$. Label the other vertices in the larger part with the odd numbers from $3$ to $n$ and the other vertices in the smaller part with the even numbers from $2$ to $n-5$. Then $(1,2,\ldots,n-1,n)$ is a Hamiltonian path of $G$. If the deleted edges are incident at a vertex in the smaller part, then without loss of generality label the deleted edges $\set{2,n-2}$ and $\set{2,n}$. Label the other vertices in the larger part with the odd numbers from $1$ to $n-4$ and the other vertices in the smaller part with the even numbers from $4$ to $n-1$. Again $(1,2,\ldots, n-1,n)$ is a Hamiltonian path of $G$.
    \end{case}
    
    \begin{case} $G$ is bipartite with part sizes differing by more than one

    Now $G$ is a subgraph of $K_{a,n-a}$ for some integer $a$ in $1 \le a \le \frac{n-2}{2}$. Thus \[e(G) \le e(K_{a,n-a}) = a(n-a) \le 
    \begin{cases}
        \frac{n-3}{2}\cdot \frac{n+3}{2} = (n^2-9)/4 &\text{for odd }n\\
        \frac{n-2}{2}\cdot \frac{n+2}{2} = n^2/4 -1 &\text{for even }n.\\
    \end{cases}\]
    The graph $K_{\floor{n/2}-1,\ceil{n/2}+1}$ is not traceable and achieves this upper bound in both the even and odd cases. Equality in the bound holds if and only if $G = K_{a,n-a}$ and $a$ is $(n-3)/2$ or $(n-2)/2$, so $K_{\floor{n/2}-1,\ceil{n/2}+1}$ is the unique extremal graph in this case. \qedhere
    \end{case}
\end{proof}

\begin{thm}
    Let $n \ge 3$. If $G$ is an $n$-vertex, $K_3$-free graph that is not Hamiltonian-connected, then \[e(G) \le e(K_{\floor{n/2},\ceil{n/2}}),\] and 
    equality holds if and only if $G = K_{\floor{n/2},\ceil{n/2}}$.
\end{thm}
\begin{proof}
    Let $G$ be such a graph. By Mantel's theorem, every $n$-vertex, $K_3$-free graph $G$ has $e(G) \le \floor{n^2/4}$, with equality if and only if $G \cong K_{\floor{n/2},\ceil{n/2}}$.
    
    If $n$ is odd then $K_{\floor{n/2},\ceil{n/2}}$ is not Hamiltonian and so not Hamiltonian-connected so gives the maximum number of edges.
    
    If $n \ge 4$ is even then $K_{n/2,n/2}$ is not Hamiltonian-connected, as there is no Hamiltonian path between any two vertices in the same partite set, so gives the maximum number of edges.
\end{proof}

\begin{thm}
    Let $0 \le k \le n-3$. Let $G$ be a $n$-vertex, $K_3$-free graph. If $G$ is not $k$-path Hamiltonian then \[e(G) \le e(K_{\ceil{n/2}-1,\floor{n/2}+1}),\]
    and this bound is tight as $K_{\ceil{n/2}-1,\floor{n/2}+1}$ is an $n$-vertex, $K_3$-free graph that is not $k$-path Hamiltonian.
\end{thm}
\begin{proof}
    Let $G$ be such a graph. By Mantel's theorem, every $n$-vertex, $K_3$-free graph $G$ has $e(G) \le \floor{n^2/4}$, with equality if and only if $G \cong K_{\floor{n/2},\ceil{n/2}}$.
    
    If $n$ is odd then $K_{\floor{n/2},\ceil{n/2}}$ is not Hamiltonian and so not $k$-path Hamiltonian so gives the maximum number of edges.

    If $n$ is even then $K_{n/2,n/2}$ is $(n-2)$-path Hamiltonian (see \cite{CK68, Kronk}) so is $k$-path Hamiltonian for every $k \le n-3$, 
    so $e(G)=n^2/4$ is not possible. As $K_{n/2-1,n/2+1}$ is not Hamiltonian and so not $k$-path Hamiltonian, $e(K_{n/2-1,n/2+1})=n^2/4-1$ gives the maximum number of edges.
\end{proof}

\begin{thm}\label{thm:r2kham}
    Let $1 \le k \le n-3$. Let $G$ be an $n$-vertex, $K_3$-free graph. If $G$ is not $k$-Hamiltonian, then 
    \[
        e(G) \le e(K_{\floor{n/2},\ceil{n/2}}),
    \]
    and this bound is tight as $K_{\floor{n/2},\ceil{n/2}}$ is an $n$-vertex, $K_3$-free graph that is not $k$-Hamiltonian.
\end{thm}

Note that $0$-Hamiltonicity is equivalent to Hamiltonicity, so the case $k=0$ is addressed in Theorem 8.1 in \cite{DK26} instead of \cref{thm:r2kham}. When $n$ is even, the answers are different for $k=0$ and $k \ge 1$.

\begin{proof}
    Let $G$ be such a graph. By Mantel's theorem, every $n$-vertex, $K_3$-free graph $G$ has $e(G) \le \floor{n^2/4}$, with equality if and only if $G \cong K_{\floor{n/2},\ceil{n/2}}$.

    If $n$ is odd then $K_{\floor{n/2},\ceil{n/2}}$ is not $k$-Hamiltonian, as the removal of a vertex from the partite set of size $\floor{n/2}$ results in a non-Hamiltonian graph, so gives the maximum number of edges. If $n$ is even then $K_{n/2,n/2}$ is not $k$-Hamiltonian, as the removal of any vertex results in a non-Hamiltonian graph, so gives the maximum number of edges. 
\end{proof}

\begin{thm}
    Let $1 \le k \le n-2$. If $G$ is an $n$-vertex, $K_3$-free graph that is not $k$-Hamiltonian-connected, then \[e(G) \le e(K_{\floor{n/2},\ceil{n/2}}),\] and equality holds if and only if $G = K_{\floor{n/2},\ceil{n/2}}$.
\end{thm}
\begin{proof}
    Let $G$ be such a graph. By Mantel's theorem, every $n$-vertex, $K_3$-free graph $G$ has $e(G) \le \floor{n^2/4}$, with equality if and only if $G \cong K_{\floor{n/2},\ceil{n/2}}$. Let $0 \le s < k$. 
    
    If $n$ is odd then $K_{\floor{n/2},\ceil{n/2}}$ is not Hamiltonian and so not Hamiltonian-connected by definition. If a graph is $k$-Hamiltonian-connected, then it is also Hamiltonian-connected, using $S = \emptyset$ in the definition. 
    For $1 \le k \le n-2$, 
    the graph $K_{\floor{n/2},\ceil{n/2}}$ is not $k$-Hamiltonian-connected and gives the maximum number of edges.
    
    If $n \ge 4$ is even then $K_{n/2,n/2}$ is not Hamiltonian-connected, as there is no Hamiltonian path between any two vertices in the same partite set. For $1 \le k \le n-2$, 
    the graph $K_{\floor{n/2},\ceil{n/2}}$ is not $k$-Hamiltonian-connected and gives the maximum number of edges.
\end{proof}

\begin{lem}[Lemma 7.2.7 in \cite{West20}]\label{minconnectivity} Deletion of an edge reduces connectivity by at most 1. If a graph $G$ is $b$-connected, and $G'$ is obtained by deleting at most $c$ edges from $G$, then $G'$ is $(b-c)$-connected.
\end{lem}

\begin{thm}\label{thm:k3conn}
    Let $1 \le k \le n-2$. If $G$ is an $n$-vertex, $K_3$-free graph that is not $k$-connected, then 
    \begin{enumerate}[(a)]
        \item for $k > \floor{n/2}$ we have $e(G) \le e(K_{\floor{n/2},\ceil{n/2}})$, and
        \item\label{part:k3conn2} for $1 \le k \le \floor{n/2}$ we have $e(G) \le e(K_{\floor{(n-1)/2},\ceil{(n-1)/2}})+k-1$,
    \end{enumerate}
and these bounds are tight as $K_{\floor{n/2},\ceil{n/2}}$ and $K_{\floor{(n-1)/2},\ceil{(n-1)/2}}$ plus a vertex of degree $k-1$ are both $n$-vertex, $K_3$-free graphs that are not $k$-connected for $k > \floor{n/2}$ and $1 \le k \le \floor{n/2}$, respectively.
\end{thm}

\begin{proof}
    For some $1 \le k \le n-2$, let $G$ be such an $n$-vertex graph. Notice that the removal of any part of a bipartite graph results in a disconnected graph and so when $k > \floor{n/2}$ the graph $K_{\floor{n/2},\ceil{n/2}}$ is not $k$-connected. By Mantel's theorem, $K_{\floor{n/2},\ceil{n/2}}$ is the unique edge-maximal graph among all $K_3$-free graphs. 
    
    We now consider when $1 \le k \le \floor{n/2}$. In this case, we show that the maximum number of edges is achieved by $K_{\floor{(n-1)/2},\ceil{(n-1)/2}}$ plus a vertex of degree $k-1$. Notice that this graph is not $k$-connected as the removal of all $k-1$ neighbors from the special vertex of degree $k-1$ results in a disconnected graph. We show that all other $n$-vertex graphs that are $K_3$-free and not $k$-connected have at most $e(K_{\floor{(n-1)/2},\ceil{(n-1)/2}})+k-1$ edges.

    \begin{case} $G$ is not bipartite

    Let $f(n)$ be the maximum number of edges in an $n$-vertex, non-bipartite, $K_3$-free graph. Note that 
    for $k\ge 3$ we have \[
    e(K_{\floor{(n-1)/2},\ceil{(n-1)/2}})+k-1 = 
    \begin{cases}
    (n-1)^2/4 +k-1 > (n-1)^2/4+1  &\text{for odd } n \\
    (n^2-2n)/4 +k-1> (n-1)^2/4+1 &\text{for even } n,\\
    \end{cases}
    \]
    and $f(n) \le (n-1)^2/4+1$ by Exercise 1.1.5 in \cite{Zhao23}, so all extremal graphs are bipartite for $k \ge 3$.
    
    For $k=2$, we show that the graphs achieving $f(n)$ edges are $k$-connected, and so we must have $e(G)<\floor{(n-1)^2/4}+1 \le e(K_{\floor{(n-1)/2},\ceil{(n-1)/2}})+k-1$. Since $G$ is not bipartite and is $K_3$-free, $G$ contains an odd cycle of length $\ell \ge 5$ (so $n \ge 5$). As shown in \cite{hardmath}, $$e(G)=\floor{(n-\ell)^2/4}+2(n-\ell)+\ell=(n-1)^2/4+1$$ only when $G$ consists of a $K_{\floor{(n-\ell)/2},\ceil{(n-\ell)/2}}$ plus an $\ell$-cycle with boundary of size $2(n-\ell)$ where $\ell=5$.
    
    Since $G$ is $K_3$-free, each vertex in the parts of $K_{\floor{(n-5)/2},\ceil{(n-5)/2}}$ has exactly two neighbors in the $5$-cycle. The removal of any vertex from the parts of $K_{\floor{(n-5)/2},\ceil{(n-5)/2}}$ or the $5$-cycle results in a connected graph. Then $G$ is $2$-connected, a contradiction, so we must have $e(G)<\floor{(n-1)^2/4}+1 \le e(K_{\floor{(n-1)/2},\ceil{(n-1)/2}})+1$.

    In the case that $k=1$, the graph $G$ is disconnected. Let $C_1, \dots, C_i$ be the connected components of $G$ with sizes $c_1 \ge  c_2 \ge \cdots \ge c_i$, respectively. Then $i\ge 2$, $\sum_{j=1}^i c_j =n$, and $c_j \le n-1$ for all $j\le i$. By Mantel's theorem, the number of edges in $G$ is maximized when each component $C_j$ is a $\T[c_j][2]$. Among such graphs $G$ where every component is a $T_2(c_j)$, notice that moving two vertices from $C_j$ to $C_1$ results in the deletion of $\floor{c_j/2}+\floor{(c_j-1)/2}$ edges and the addition of $\ceil{c_1/2}+\ceil{(c_1+1)/2}$ edges for a net increase in the number of edges since $c_1 \ge c_j$, so the extremal graphs have $c_i \le \cdots \le c_2 \le 1$. If $C_j$ consists of a single isolated vertex, then the removal of the vertex from $C_j$ and the addition of the vertex to $C_1$ strictly increases the number of edges in $G$. The number of edges in $G$ is then maximized when $c_1=n-1$ and $c_2=1$. Then $G$ consists of a $K_{\floor{(n-1)/2},\ceil{(n-1)/2}}$ plus an isolated vertex and $e(G) = e(K_{\floor{(n-1)/2},\ceil{(n-1)/2}})\le (n-1)^2/4$.
    \end{case}

    \begin{case} $G$ is a subgraph of  $K_{\floor{n/2},\ceil{n/2}}$

    The graph $K_{\floor{n/2},\ceil{n/2}}$ is $k$-connected for $1 \le k \le \floor{n/2}$ as the removal of any $k-1$ vertices would result in a complete bipartite graph with nonempty parts. By \cref{minconnectivity}, deleting any $\floor{n/2}-k$ edges from $K_{\floor{n/2},\ceil{n/2}}$ results in a $k$-connected bipartite graph, thus $G$ is obtained by deleting at least $\floor{n/2}-k+1$ edges of $K_{\floor{n/2},\ceil{n/2}}$, so \begin{align*}
        e(G) &\le e(K_{\floor{n/2},\ceil{n/2}})-\floor{n/2}+k-1 \\
        &= e(K_{\floor{(n-1)/2},\ceil{(n-1)/2}})+k-1. 
    \end{align*}    
    \end{case}

    \begin{case} $G$ is bipartite with part sizes differing by more than one
    
    In this case, $G$ is a subgraph of $K_{a,n-a}$ for some integer $a$ in $1 \le a \le \frac{n-2}{2}$. Thus $e(G) \le e(K_{a,n-a}) = a(n-a)$. We show $e(G) \le e(K_{\floor{(n-1)/2},\ceil{(n-1)/2}})+k-1$ by considering two cases depending on how $k$ compares to $a$. 
    
    For $k>a$, notice that the removal of any part results in a disconnected graph and so we have $K_{a,n-a}$ is not $k$-connected. For $a \le (n-2)/2$, we have
    \[
        4\parens[\bigg]{\frac{n^2-2n}{4}+a - a(n-a)} = (n-2a-2)(n-2a) \ge 0,
    \]
    with equality if and only if $a = (n-2)/2$, so, since $a \le k-1$, 
    \begin{align*}
        e(K_{a,n-a})=a(n-a) &\le \begin{cases}
            (n-1)^2/4 +k-1  &\text{for odd } n \\
            (n^2-2n)/4 +k-1 &\text{for even } n \\
        \end{cases}\\
        &= e(K_{\floor{(n-1)/2},\ceil{(n-1)/2}})+k-1,
    \end{align*} with equality if and only if $a = (n-2)/2 = k-1$.

    For $1 \le k \le a$, the graph $K_{a,n-a}$ is $a$-connected. By \cref{minconnectivity} at least $a-k+1$ edges need to be removed from $G$ to yield a graph that is not $k$-connected and so \begin{align*}
        e(G) &\le a(n-a)-(a-k+1) = -a^2+(n-1)a+k-1\\
        &< e(K_{\floor{(n-1)/2},\ceil{(n-1)/2}})+k-1
    \end{align*} edges, where the last inequality follows from maximizing the concave-down quadratic function of $a$ by substituting in $a=\floor{(n-2)/2}$.\qedhere
    \end{case}
\end{proof}

\section{Clique Density Conditions}\label{sec:clique}

In this section we extend our edge density conditions to clique density conditions via an observation about the extremal graphs. 
First, we use a similar observation to obtain best-possible clique density conditions for Hamiltonicity-like properties in not necessarily $K_{r+1}$-free graphs.

Recall that the \emph{colex} (or \emph{colexicographic}) \emph{order} on finite subsets of $\N$, denoted by $<_C$, is defined by $A <_C B$ if and only if $\max(A\symd B) \in B$. The \emph{colex graph on $m$ edges}, denoted by $\cC(m)$, has edge set consisting of the first $m$ pairs of natural numbers in colex order. Its vertex set is the union of all of the edges, i.e., the subset of $\N$ needed to support the edges and avoid isolated vertices.

The colex graph equivalently is the graph on $m$ edges consisting of the largest complete graph $K_p$ that can fit on $m$ edges together with one additional vertex whose degree is the remaining number of edges, $m-\binom{p}{2}$. It follows from the Kruskal-Katona Theorem that the colex graph on $m$ edges has the maximum number of $t$-cliques among all graphs on at most $m$ edges. For a proof of this implication, see \cite{KR23}.

\begin{thm}[Corollary of Kruskal-Katona Theorem \cite{Katona, Kruskal}]\label{thm:KK}
    For every $t \ge 2$, if $G$ is a graph on at most $m$ edges, then $\k(G) \le \k(C(m))$.
\end{thm}

One immediate consequence of \cref{thm:KK} is as follows.
\begin{lem}\label{thm:generalcolex}
    Let $\mathcal{A}$ be a set of $n$-vertex graphs. Let $m = \max\set{e(G):G\in\mathcal{A}}$. Let $G \in \mathcal{A}$. For every $t \ge 2$,
    \[
        \k(G) \le \k(\cC(m)),
    \]
    and if $\cC(m) \in \mathcal{A}$ then this upper bound is tight as equality holds if $G \cong \cC(m)$.
\end{lem}

\begin{proof}
    Let $G \in \mathcal{A}$. Then $e(G) \le m$ by the definition of $m$. By \cref{thm:KK}, for all $t \ge 2$, we have $\k(G) \le \k(C(m))$. The condition $\cC(m) \in \mathcal{A}$ ensures that this upper bound is achieved.
\end{proof}

Recall that the extremal graphs in Theorem \ref{thm:ore} and the analogous theorems for the other properties implied by \cref{thm:kstableedge} all consist of a complete graph together with one more vertex of a given degree, so they are colex graphs. Thus, as a consequence of these theorems and \cref{thm:generalcolex}, we obtain clique conditions sufficient to imply the same properties.

\begin{cor}\label{cor:ore}
    Let $G$ be a graph on $n$ vertices. Let $t \ge 2$.
    \begin{enumerate}[(a)]
        \item If $G$ is not traceable, then $\k(G) \le \k(\cC(\binom{n-1}{2}))$. Equality holds if $G$ is a $\cC(\binom{n-1}{2})\cong K_{n-1}$ plus an isolated vertex.
        \item\label{part:cororeham} If $G$ is not Hamiltonian, then $\k(G) \le \k(\cC(\binom{n-1}{2}+1))$. For $n \ge 2$, equality holds if $G \cong \cC(\binom{n-1}{2}+1)$, which is a $K_{n-1}$ plus a pendant edge.
        \item If $G$ is not Hamiltonian-connected, then $\k(G) \le \k(\cC(\binom{n-1}{2}+2))$. For $n \ge 3$, equality holds if $G \cong \cC(\binom{n-1}{2}+2)$, which is a $K_{n-1}$ plus a vertex of degree $2$.
        \item If $G$ is not $k$-path Hamiltonian for some $0 \le k \le n-3$, then $\k(G) \le \k(C(\binom{n-1}{2}+k+1))$. 
        Equality holds if $G \cong \cC(\binom{n-1}{2}+k+1)$.
        \item\label{part:kHamclique} If $G$ is not $k$-Hamiltonian for some $0 \le k \le n-3$, then $\k(G) \le \k(C(\binom{n-1}{2}+k+1))$. 
        Equality holds if $G \cong \cC(\binom{n-1}{2}+k+1)$.
        \item If $G$ is not $k$-Hamiltonian-connected for some $1 \le k \le n-3$, then $\k(G) \le \k(C(\binom{n-1}{2}+k+1))$. 
        Equality holds if $G \cong \cC(\binom{n-1}{2}+k+1)$.
        \item If $G$ is not $k$-connected for some $1 \le k \le n-1$, then $\k(G) \le \k(C(\binom{n-1}{2}+k-1))$. 
        Equality holds if $G \cong \cC(\binom{n-1}{2}+k-1)$.
    \end{enumerate}
\end{cor}

\begin{proof}
    Let $\A$ be the set of $n$-vertex graphs which are not traceable, Hamiltonian, Hamiltonian-connected, $k$-path Hamiltonian, $k$-Hamiltonian, $k$-Hamiltonian-connected, or $k$-connected for parts (a) through (g), respectively. By \cref{table}, \cref{thm:kstableedge}, and \cref{rem:colexproperties}, the values of $m$ (as defined in \cref{thm:generalcolex}) are $\binom{n-1}{2}$, $\binom{n-1}{2}+1$, $\binom{n-1}{2}+2$, $\binom{n-1}{2}+k+1$, $\binom{n-1}{2}+k+1$, $\binom{n-1}{2}+k+1$, and $\binom{n-1}{2}+k-1$, respectively. In each case we have $\cC(m) \in \mathcal{A}$. Then \cref{thm:generalcolex} implies that $\k(G) \le \k(\cC(m))$ with equality if $G \cong \cC(m)$.
\end{proof}

Parts (a) and (b) of \cref{cor:ore} also were special cases of Theorems 1.6 and 1.7 of Chakraborti and Chen \cite{ChakrabortiChen}. Part (d) is closely related to a corollary of a theorem of F{\"u}redi, Kostochka, and Luo \cite{FKL19} on the maximum number of $t$-cliques in graphs of minimum degree at least $d > \ell$ which do not have the property that every linear forest on $\ell$ edges is contained in a Hamiltonian cycle.

In \cite{DK26} we found the maximum number of $t$-cliques in $K_{r+1}$-free non-Hamiltonian graphs by observing that the graphs $\Gs$ are \emph{$r$-partite colex Tur\'{a}n graphs} (graphs consisting of the first $m$ edges in colex order from a complete $r$-partite graph on vertex set $\N$) and using a theorem of Frohmader \cite{Frohmader} analogously to how the Kruskal-Katona Theorem is used above. Now that we have \cref{thm:cliquestability}, we can find clique density conditions in $K_{r+1}$-free graphs for the other properties. The theorem for Hamiltonicity and chorded pancyclicity (Corollary 6.4 and Corollary 7.3 in \cite{DK26}) is included as part \ref{part:Hamclique} in the theorem statement below for completeness.

\begin{cor}\label{cor:cliques}
    Let $G$ be an $n$-vertex, $K_{r+1}$-free graph where $r \ge 3$. Let $t \ge 2$.
\begin{enumerate}[(a)]
\item If $G$ is not traceable and \[n \ge \begin{cases}20 & \text{if }r=3\\ 1 & \text{if } r\ge4 
\end{cases},\quad\text{then } \k(G) \leq \k(\Gs[r][n][-1]).\] Equality holds if $G \cong \Gs[r][n][-1]$, which is $\T[n-1][r]$ plus an isolated vertex.

\item\label{part:Hamclique} If $G$ is not Hamiltonian (or, for $n\ge 4$, if $G$ is not chorded pancyclic) and \[n \ge \begin{cases}26 & \text{if }r=3\\ 11 & \text{if } r=4\\ 
2 & \text{if } r \ge 5\end{cases},\quad\text{then } \k(G) \leq \k(\Gs[r][n][0]).\]
Equality holds if $G \cong \Gs[r][n][0]$, which is $\T[n-1][r]$ plus a vertex of degree $1$.

\item If $G$ is not Hamiltonian-connected and \[n \ge \begin{cases} 32 & \text{if } r=3\\ 16 & \text{if } r=4\\ 10 & \text{if } r=5\\
4 & \text{if } r \ge 6\end{cases},\quad\text{then }\k(G) \leq \k(\Gs[r][n][1]).\] 
Equality holds if $G \cong \Gs[r][n][1]$, which is $\T[n-1][r]$ plus a vertex of degree $2$ whose neighbors are adjacent.

\item If $G$ is not $k$-path Hamiltonian and \[n \ge \begin{cases} 6k+26 & \text{if } r=3\\ 5k+11 & \text{if } 4 \le r \le 7\\  k+5+ 4(k+4)/(r-4) & \text{if } r \ge 8\end{cases},\quad\text{then } \k(G) \leq \k(\Gs[r][n][k]).\] 
Equality holds if $G \cong \Gs[r][n][k]$.

\item If $G$ is not $k$-Hamiltonian and \[n \ge \begin{cases} 6k+26 & \text{if } r=3\\ 5k+11 & \text{if } 4 \le r \le 7\\  k+5+ 4(k+4)/(r-4) & \text{if } r \ge 8\end{cases},\quad\text{then }\k(G) \leq \k(\Gs[r][n][k]).\]
Equality holds if $G \cong \Gs[r][n][k]$.

\item If $G$ is not $k$-Hamiltonian-connected and \[n \ge \begin{cases} 6k+26 & \text{if } r=3\\ 5k+11 & \text{if } r=4\\ 3k+7 & \text{if } 5 \le r \le 7\\ k+5+4(k+4)/(r-4) & \text{if } r \ge 8\end{cases},\quad\text{then }\k(G) \leq \k(\Gs[r][n][k]).\] 
Equality holds if $G \cong \Gs[r][n][k]$.

\item If $G$ is not $k$-connected and \[n \ge \begin{cases} 6k+14 & \text{if } r=3\\ 5k+1 & \text{if } 4 \le r \le 7\\ 2k+5 & \text{if } r \ge 8\end{cases},\quad\text{then }\k(G) \leq \k(\Gs[r][n][k-2]).\] 
Equality holds if $G \cong \Gs[r][n][k-2]$.
\end{enumerate}
\end{cor}

\begin{proof}
    Apply \cref{thm:cliquestability} for each property using \cref{table} and \cref{thm:kstable}. By \cref{prop:gell}, in each part, the given graph avoids the forbidden property as required.
\end{proof}

\cref{cor:cliques} alternatively can be viewed as a consequence of the following theorem and observation.

\begin{thm}[Theorem 6.2 in \cite{DK26}]\label{thm:generalcolexturan}
    Let $\mathcal{A}$ be a set of $n$-vertex, $K_{r+1}$-free graphs. Let $m = \max\set{e(G):G\in \mathcal{A}}$. Let $G \in \mathcal{A}$. For every $t \ge 2$, $\k(G) \le \k(\CT)$, and if $\CT\in\mathcal{A}$ then this upper bound is tight as equality holds if $G \cong \CT$.
\end{thm}

\begin{obs}[Observation 6.3 in \cite{DK26}]\label{obs:colexTuran} For every $r$ and $m$, the $r$-partite colex Tur\'{a}n graph $\CT$ on $m$ edges is $\Gs[r][n][\ell] \in \Gell$, where $n = \abs{V(\CT)}$ and $\ell = \delta(\CT)-1$.
\end{obs}

\section{Open Problems}\label{sec:open}

Relaxing the non-Hamiltonicity condition to an upper bound on the circumference, or relaxing the non-traceability condition to a $P_k$-free condition, yields the following two open questions.

\begin{qu}\label{q:cycle}What is the maximum number of edges (or $t$-cliques) in $n$-vertex, $K_{r+1}$-free graphs of circumference at most $k \le n-1$?\end{qu} 

\begin{qu}\label{q:path}What is the maximum number of edges (or $t$-cliques) in $n$-vertex, $\set{K_{r+1},P_k}$-free graphs for $k \le n-1$?\end{qu}

\cref{q:cycle} may be quite challenging, based on the difficulty of the weaker problem studied in \cite{Araujo}. If we remove the $K_{r+1}$-free condition from Questions \ref{q:cycle} and \ref{q:path}, then the resulting questions have been answered, first asymptotically by Luo \cite{Luo} and then exactly by Chakraborti and Chen \cite{ChakrabortiChen}. Moreover, Luo \cite{Luo} determined asymptotically the maximum number of $t$-cliques among $n$-vertex graphs which are $2$-connected and do not contain long cycles or long paths. Adding a $2$-connectedness condition to our problem would force the extremal graphs to be non-Hamiltonian without having a pendant edge. Similarly, the graphs in $\Gell$ are not $(\ell+2)$-connected, as the neighborhood of the exceptional vertex of degree $\ell+1$ is a separating set. Therefore adding an $(\ell+2)$-connectedness condition would result in different extremal graphs.

\begin{qu}\label{qu:2conn}What is the maximum number of edges in $n$-vertex, $K_{r+1}$-free, 
\begin{enumerate}[(a)]
    \item non-traceable graphs which are connected?
    \item non-Hamiltonian graphs which are $2$-connected?
    \item non-Hamiltonian-connected graphs which are $3$-connected?
    \item non-$k$-Hamiltonian graphs which are $(k+2)$-connected?
    \item non-$k$-Hamiltonian-connected graphs which are $(k+2)$-connected?
\end{enumerate}\end{qu}

We did not attempt to extend Theorems \ref{theorem:kr+1all} and \ref{theorem:kr+1kpath1} to maximize the number of edges for smaller values of $n$, where the graph $\Gs$ is not always extremal. \cref{table}, \cref{thm:kstableedge}, and \cref{rem:colexproperties} address $n \le r+1$, as the extremal graphs all are $K_{r+1}$-free. It would also be interesting to characterize the extremal graphs for the problem of maximizing the number of $t$-cliques.

\begin{qu}
    Which graphs not in $\Gell$, if any, achieve the maximum number of $t$-cliques among $n$-vertex, $K_{r+1}$-free, non-Hamiltonian graphs? Among graphs which are not traceable, Hamiltonian, $k$-path Hamiltonian, $k$-Hamiltonian, $k$-Hamiltonian-connected, or $k$-connected?
\end{qu}

\section*{Acknowledgments}

The authors thank V\'{a}clav Chv\'{a}tal and Linda Lesniak for clarification of the discussion of stable properties in \cite{BC76}, Jamie Radcliffe for comments which improved the presentation of the paper, and Zhanar Berikkyzy, Kirsten Hogenson, and Jessica McDonald for helpful discussions of stable properties. The first author is supported in part by the NSF Grant DMS-2402312. The second author is supported in part by Simons Foundation Grant MP-TSM-00002688.

\bibliographystyle{plainurl}
\bibliography{references}

\begin{thebibliography}{10}

\bibitem{Adamus}
J.~Adamus.
\newblock Edge condition for {H}amiltonicity in balanced tripartite graphs.
\newblock {\em Opuscula Math.}, 29(4):337--343, 2009.
\newblock \href {https://doi.org/10.7494/OpMath.2009.29.4.337}
  {\path{doi:10.7494/OpMath.2009.29.4.337}}.

\bibitem{AdamusBalanced}
J.~Adamus and L.~Adamus.
\newblock Ore and {E}rd{\H{o}}s type conditions for long cycles in balanced
  bipartite graphs.
\newblock {\em Discrete Math. Theor. Comput. Sci.}, 11(2):57--69, 2009.

\bibitem{AdamusUnbalanced}
L.~Adamus.
\newblock Edge condition for long cycles in bipartite graphs.
\newblock {\em Discrete Math. Theor. Comput. Sci.}, 11(2):25--32, 2009.

\bibitem{AmarEtAl}
Denise Amar, Odile Favaron, Pedro Mago, and Oscar Ordaz.
\newblock Biclosure and bistability in a balanced bipartite graph.
\newblock {\em J. Graph Theory}, 20(4):513--529, 1995.
\newblock \href {https://doi.org/10.1002/jgt.3190200414}
  {\path{doi:10.1002/jgt.3190200414}}.

\bibitem{Araujo}
Gabriela {Araujo-Pardo}, Zhanar Berikkyzy, Jill Faudree, Kirsten Hogenson,
  Rachel Kirsch, Linda Lesniak, and Jessica McDonald.
\newblock Finding long cycles in balanced tripartite graphs: {{A}} first step.
\newblock In Daniela Ferrero, Leslie Hogben, Sandra~R. Kingan, and Gretchen~L.
  Matthews, editors, {\em Research Trends in Graph Theory and Applications},
  pages 1--10. {Springer International Publishing}, {Cham}, 2021.
\newblock \href {https://doi.org/10.1007/978-3-030-77983-2_1}
  {\path{doi:10.1007/978-3-030-77983-2_1}}.

\bibitem{BaggaVarma}
J.~S. Bagga and B.~N. Varma.
\newblock Hamiltonian properties in bipartite graphs.
\newblock {\em Bull. Inst. Combin. Appl.}, 26:71--85, 1999.

\bibitem{BV91}
Kunwarjit~S. Bagga and Badri~N. Varma.
\newblock Bipartite graphs and degree conditions.
\newblock In {\em Graph theory, combinatorics, algorithms, and applications
  ({S}an {F}rancisco, {CA}, 1989)}, pages 564--573. SIAM, Philadelphia, PA,
  1991.

\bibitem{BBHKNSWY15}
D.~Bauer, H.~J. Broersma, J.~van~den Heuvel, N.~Kahl, A.~Nevo, E.~Schmeichel,
  D.~R. Woodall, and M.~Yatauro.
\newblock Best {{Monotone Degree Conditions}} for {{Graph Properties}}: {{A
  Survey}}.
\newblock {\em Graphs and Combinatorics}, 31(1):1--22, January 2015.
\newblock \href {https://doi.org/10.1007/s00373-014-1465-6}
  {\path{doi:10.1007/s00373-014-1465-6}}.

\bibitem{Berge}
Claude Berge.
\newblock {\em Graphs and hypergraphs}.
\newblock North-Holland Mathematical Library, Vol. 6. North-Holland Publishing
  Co., Amsterdam-London; American Elsevier Publishing Co., Inc., New York,
  revised edition, 1976.
\newblock Translated from the French by Edward Minieka.

\bibitem{BHKM26}
Zhanar Berikkyzy, Kirsten Hogenson, Rachel Kirsch, and Jessica McDonald.
\newblock Maximizing the number of stars in graphs with forbidden properties.
\newblock {\em Discrete Mathematics}, 349(2):114846, 2026.
\newblock URL:
  \url{https://www.sciencedirect.com/science/article/pii/S0012365X25004546},
  \href {https://doi.org/10.1016/j.disc.2025.114846}
  {\path{doi:10.1016/j.disc.2025.114846}}.

\bibitem{Boesch74}
F.T Boesch.
\newblock The strongest monotone degree condition for n-connectedness of a
  graph.
\newblock {\em Journal of Combinatorial Theory, Series B}, 16(2):162--165,
  1974.
\newblock \href {https://doi.org/10.1016/0095-8956(74)90058-6}
  {\path{doi:10.1016/0095-8956(74)90058-6}}.

\bibitem{Bondy69}
J.~A. Bondy.
\newblock Properties of graphs with constraints on degrees.
\newblock {\em Studia Sci. Math. Hungar.}, 4:473--475, 1969.

\bibitem{BC76}
J.A. Bondy and V.~Chv\'{a}tal.
\newblock A method in graph theory.
\newblock {\em Discrete Mathematics}, 15(2):111--135, 1976.
\newblock \href {https://doi.org/10.1016/0012-365X(76)90078-9}
  {\path{doi:10.1016/0012-365X(76)90078-9}}.

\bibitem{BS81}
Richard~A. Brualdi and Robert~F. Shanny.
\newblock Hamiltonian line graphs.
\newblock {\em J. Graph Theory}, 5(3):307--314, 1981.
\newblock \href {https://doi.org/10.1002/jgt.3190050312}
  {\path{doi:10.1002/jgt.3190050312}}.

\bibitem{ChakrabortiChen}
Debsoumya Chakraborti and Da~Qi Chen.
\newblock Exact results on generalized {Erd\H{o}s-Gallai} problems.
\newblock {\em European Journal of Combinatorics}, 120:103955, 2024.
\newblock \href {https://doi.org/10.1016/j.ejc.2024.103955}
  {\path{doi:10.1016/j.ejc.2024.103955}}.

\bibitem{CKK68}
Gary Chartrand, S.~F. Kapoor, and Hudson~V. Kronk.
\newblock A sufficient condition for {\emph{n}}-connectedness of graphs.
\newblock {\em Mathematika}, 15(1):51--52, June 1968.
\newblock \href {https://doi.org/10.1112/S0025579300002369}
  {\path{doi:10.1112/S0025579300002369}}.

\bibitem{CKL70}
Gary Chartrand, S.~F. Kapoor, and Don~R. Lick.
\newblock {$n$}-{H}amiltonian graphs.
\newblock {\em J. Combinatorial Theory}, 9:308--312, 1970.
\newblock \href {https://doi.org/10.1016/S0021-9800(70)80069-2}
  {\path{doi:10.1016/S0021-9800(70)80069-2}}.

\bibitem{CK68}
Gary Chartrand and Hudson~V. Kronk.
\newblock Randomly traceable graphs.
\newblock {\em SIAM Journal on Applied Mathematics}, 16(4):696--700, 1968.
\newblock URL: \url{http://www.jstor.org/stable/2099120}.

\bibitem{CFGJL}
G.~Chen, R.~J. Faudree, R.~J. Gould, M.~S. Jacobson, and L.~Lesniak.
\newblock Hamiltonicity in balanced $k$-partite graphs.
\newblock {\em Graphs and Combinatorics}, 11(3):221--231, 1995.
\newblock \href {https://doi.org/10.1007/BF01793008}
  {\path{doi:10.1007/BF01793008}}.

\bibitem{ChenJacobson}
G.~Chen and M.~S. Jacobson.
\newblock Degree sum conditions for {H}amiltonicity on {$k$}-partite graphs.
\newblock {\em Graphs Combin.}, 13(4):325--343, 1997.
\newblock \href {https://doi.org/10.1007/BF03353011}
  {\path{doi:10.1007/BF03353011}}.

\bibitem{Chvatal}
V~Chv{\'a}tal.
\newblock On {{Hamilton}}'s ideals.
\newblock {\em Journal of Combinatorial Theory, Series B}, 12(2):163--168,
  April 1972.
\newblock \href {https://doi.org/10.1016/0095-8956(72)90020-2}
  {\path{doi:10.1016/0095-8956(72)90020-2}}.

\bibitem{DK26}
Aleyah {Dawkins} and Rachel {Kirsch}.
\newblock {Ore plus Tur{\'a}n}.
\newblock {\em arXiv e-prints}, page arXiv:2310.11452, October 2023.
\newblock \href {https://arxiv.org/abs/2310.11452} {\path{arXiv:2310.11452}},
  \href {https://doi.org/10.48550/arXiv.2310.11452}
  {\path{doi:10.48550/arXiv.2310.11452}}.

\bibitem{Dirac}
G.~A. Dirac.
\newblock Some theorems on abstract graphs.
\newblock {\em Proc. London Math. Soc. (3)}, 2:69--81, 1952.

\bibitem{ES88}
R.~C. Entringer and E.~F. Schmeichel.
\newblock Edge conditions and cycle structure in bipartite graphs.
\newblock {\em Ars Combin.}, 26:229--232, 1988.

\bibitem{Erdos}
Paul~L. Erd\H{o}s and Tibor Gallai.
\newblock On maximal paths and circuits of graphs.
\newblock {\em Acta Mathematica Academiae Scientiarum Hungarica}, 10:337--356,
  1959.
\newblock \href {https://doi.org/10.1007/BF02024498}
  {\path{doi:10.1007/BF02024498}}.

\bibitem{FerreroLesniak}
D.~Ferrero and L.~Lesniak.
\newblock Chorded pancyclicity in {$k$}-partite graphs.
\newblock {\em Graphs Combin.}, 34(6):1565--1580, 2018.
\newblock \href {https://doi.org/10.1007/s00373-018-1942-4}
  {\path{doi:10.1007/s00373-018-1942-4}}.

\bibitem{Frohmader}
Andrew Frohmader.
\newblock Face vectors of flag complexes.
\newblock {\em Israel Journal of Mathematics}, 164(1):153--164, March 2008.
\newblock \href {https://doi.org/10.1007/s11856-008-0024-3}
  {\path{doi:10.1007/s11856-008-0024-3}}.

\bibitem{FKL18}
Zolt\'an F\"uredi, Alexandr Kostochka, and Ruth Luo.
\newblock Extensions of a theorem of {Erd\H{o}s} on nonhamiltonian graphs.
\newblock {\em J. Graph Theory}, 89(2):176--193, 2018.
\newblock \href {https://doi.org/10.1002/jgt.22246}
  {\path{doi:10.1002/jgt.22246}}.

\bibitem{FKL19}
Zolt{\'a}n F{\"u}redi, Alexandr Kostochka, and Ruth Luo.
\newblock A variation of a theorem by {{P\'osa}}.
\newblock {\em Discrete Mathematics}, 342(7):1919--1923, July 2019.
\newblock \href {https://doi.org/10.1016/j.disc.2019.03.008}
  {\path{doi:10.1016/j.disc.2019.03.008}}.

\bibitem{GP25}
D{\'a}niel {Gerbner} and Cory {Palmer}.
\newblock {Survey of generalized Tur{\'a}n problems -- counting subgraphs}.
\newblock {\em arXiv e-prints}, page arXiv:2506.03418, June 2025.
\newblock \href {https://arxiv.org/abs/2506.03418} {\path{arXiv:2506.03418}}.

\bibitem{G03}
Ronald~J. Gould.
\newblock Advances on the {H}amiltonian problem---a survey.
\newblock {\em Graphs Combin.}, 19(1):7--52, 2003.
\newblock \href {https://doi.org/10.1007/s00373-002-0492-x}
  {\path{doi:10.1007/s00373-002-0492-x}}.

\bibitem{G14}
Ronald~J. Gould.
\newblock Recent advances on the {H}amiltonian problem: {S}urvey {III}.
\newblock {\em Graphs Combin.}, 30(1):1--46, 2014.
\newblock \href {https://doi.org/10.1007/s00373-013-1377-x}
  {\path{doi:10.1007/s00373-013-1377-x}}.

\bibitem{hardmath}
hardmath (https://math.stackexchange.com/users/3111/hardmath).
\newblock Determining the maximum number of edges in triangle-free graphs that
  are not bipartite.
\newblock Mathematics Stack Exchange.
\newblock URL:
  \url{https://math.stackexchange.com/questions/4202442/prove-that-fn-le-frac14n-121}.

\bibitem{Lavrov}
Misha~Lavrov (https://math.stackexchange.com/users/383078/misha lavrov).
\newblock Alternative way to calculate number of edges in {T}urán-graphs?
\newblock Mathematics Stack Exchange.
\newblock URL: \url{https://math.stackexchange.com/q/4090101}.

\bibitem{Katona}
G.~Katona.
\newblock A theorem of finite sets.
\newblock In {\em Theory of graphs ({P}roc. {C}olloq., {T}ihany, 1966)}, pages
  187--207. Academic Press, New York, 1968.
\newblock \href {https://doi.org/10.1007/978-0-8176-4842-8_27}
  {\path{doi:10.1007/978-0-8176-4842-8_27}}.

\bibitem{KR23}
Rachel Kirsch and Jamie Radcliffe.
\newblock Many cliques in bounded-degree hypergraphs.
\newblock {\em SIAM Journal on Discrete Mathematics}, 37(3):1436--1456, 2023.
\newblock \href {https://doi.org/10.1137/22M1507565}
  {\path{doi:10.1137/22M1507565}}.

\bibitem{KronkGen}
Hudson~V. Kronk.
\newblock Generalization of a theorem of {P\'{o}sa}.
\newblock {\em Proc. Amer. Math. Soc.}, 21(1):77--78, April 1969.
\newblock \href {https://doi.org/10.1090/S0002-9939-1969-0237377-9}
  {\path{doi:10.1090/S0002-9939-1969-0237377-9}}.

\bibitem{Kronk}
Hudson~V. Kronk.
\newblock A note on {$k$-path Hamiltonian graphs}.
\newblock {\em Journal of Combinatorial Theory}, 7(2):104--106, September 1969.
\newblock \href {https://doi.org/10.1016/S0021-9800(69)80043-8}
  {\path{doi:10.1016/S0021-9800(69)80043-8}}.

\bibitem{Kruskal}
Joseph~B. Kruskal.
\newblock 12. {{The Number}} of {{Simplices}} in a {{Complex}}.
\newblock In Richard Bellman, editor, {\em Mathematical {{Optimization
  Techniques}}}, pages 251--278. {University of California Press}, December
  1963.
\newblock \href {https://doi.org/10.1525/9780520319875-014}
  {\path{doi:10.1525/9780520319875-014}}.

\bibitem{KO14}
Daniela K\"uhn and Deryk Osthus.
\newblock Hamilton cycles in graphs and hypergraphs: an extremal perspective.
\newblock In {\em Proceedings of the {I}nternational {C}ongress of
  {M}athematicians---{S}eoul 2014. {V}ol. {IV}}, pages 381--406. Kyung Moon Sa,
  Seoul, 2014.

\bibitem{Lick}
Don~R. Lick.
\newblock {$n$-hamiltonian connected graphs}.
\newblock {\em Duke Mathematical Journal}, 37(2):387 -- 392, 1970.
\newblock \href {https://doi.org/10.1215/S0012-7094-70-03749-X}
  {\path{doi:10.1215/S0012-7094-70-03749-X}}.

\bibitem{Luo}
Ruth Luo.
\newblock The maximum number of cliques in graphs without long cycles.
\newblock {\em Journal of Combinatorial Theory, Series B}, 128:219--226,
  January 2018.
\newblock \href {https://doi.org/10.1016/j.jctb.2017.08.005}
  {\path{doi:10.1016/j.jctb.2017.08.005}}.

\bibitem{Mantel}
W.~Mantel.
\newblock Problem 28 ({Solution by H. Gouwentak, W. Mantel, J. Teixeira de
  Mattes, F. Schuh and W. A. Wythoff}.
\newblock {\em Wiskundige Opgaven}, 10:60--61, 1907.

\bibitem{MoonMoser}
J.~Moon and L.~Moser.
\newblock On hamiltonian bipartite graphs.
\newblock {\em Israel Journal of Mathematics}, 1(3):163--165, 1963.

\bibitem{Ore}
O.~Ore.
\newblock Note on {H}amilton circuits.
\newblock {\em Amer. Math. Monthly}, 67:55, 1960.

\bibitem{OreEdgeCond}
O.~Ore.
\newblock Arc coverings of graphs.
\newblock {\em Ann. Mat. Pura Appl. (4)}, 55:315--321, 1961.
\newblock \href {https://doi.org/10.1007/BF02412090}
  {\path{doi:10.1007/BF02412090}}.

\bibitem{OreHConnected}
Oystein Ore.
\newblock Hamilton connected graphs.
\newblock {\em J. Math. Pures Appl. (9)}, 42:21--27, 1963.

\bibitem{turan}
P.~Tur\'{a}n.
\newblock On an extremal problem in graph theory.
\newblock {\em Matematikai és Fizikai Lapok}, 48:436--452, 1941.

\bibitem{West20}
Douglas~Brent West.
\newblock {\em Combinatorial Mathematics}.
\newblock Cambridge University Press, Cambridge, UK ; New York, NY, 2020.

\bibitem{Williamson}
James~E. Williamson.
\newblock {\em On Hamiltonian-connected graphs}.
\newblock PhD thesis, Western Michigan University, 1973.
\newblock URL: \url{https://scholarworks.wmich.edu/dissertations/2851/}.

\bibitem{Zhao23}
Yufei Zhao.
\newblock {\em Graph Theory and Additive Combinatorics: Exploring Structure and
  Randomness}.
\newblock Cambridge University Press, 2023.

\bibitem{Zykov}
A.~A. Zykov.
\newblock On some properties of linear complexes.
\newblock {\em Mat. Sbornik N.S.}, 24/66:163--188, 1949.

\end{thebibliography}
\end{document}